\documentclass[10pt,reqno]{amsart}
\usepackage[utf8]{inputenc}
\usepackage[english]{babel}
\usepackage{amsmath, amssymb, amsthm}
\usepackage{mathrsfs}
\usepackage[shortlabels]{enumitem}
\usepackage{graphicx}
\usepackage{tikz-cd}
\usepackage{xcolor}
\usepackage{mathtools}
\numberwithin{equation}{section} %para numerar ecuaciones por sección

\usepackage{diagbox} % to have the diagonal separation on the 1,1 entry on a table
\usepackage{tensor}  %para escribir tensores y que los superíndices no queden debajo de los subíndices

\mathtoolsset{showonlyrefs=true}
\usepackage{hyperref} %this have to be load BEFORE amsrefs

\usepackage[alphabetic]{amsrefs} %alph is for references as [NameYear]
\theoremstyle{plain}
\newtheorem{thm}{Theorem}[section]
\newtheorem{prop}[thm]{Proposition}
\newtheorem{lemma}[thm]{Lemma}

\newtheorem{cor}[thm]{Corollary}

\theoremstyle{definition}
\newtheorem{defin}[thm]{Definition}

\newtheorem{rmk}[thm]{Remark}

\allowdisplaybreaks

\newcommand{\Z}{\mathbb{Z}}
\newcommand{\R}{\mathbb{R}}
\newcommand{\C}{\mathbb{C}}

\DeclareMathOperator{\supp}{supp}

\newcommand{\eps}{\varepsilon}

\DeclareMathOperator{\real}{Re}

\DeclareMathOperator{\vol}{Vol}

\newcommand{\A}{\mathbb{A}}

\DeclareMathOperator{\tr}{tr}

\def\XXint#1#2#3{{\setbox0=\hbox{$#1{#2#3}{\int}$ }
\vcenter{\hbox{$#2#3$ }}\kern-.6\wd0}}

\usepackage{scalerel,stackengine}
\stackMath
\newcommand\hhat[1]{%
\savestack{\tmpbox}{\stretchto{%
  \scaleto{%
    \scalerel*[\widthof{\ensuremath{#1}}]{\kern.1pt\mathchar"0362\kern.1pt}%
    {\rule{0ex}{\textheight}}%WIDTH-LIMITED CIRCUMFLEX
  }{\textheight}% 
}{2.4ex}}%
\stackon[-6.9pt]{#1}{\tmpbox}%
}
\usepackage{seans-shortcuts}

\title[Guillarmou's Normal Operator]{Guillarmou's Normal Operator\\ for Magnetic and Thermostat Flows}
\author{Sebastián Muñoz-Thon}
\address{Universit\'e Paris-Saclay, Laboratoire de math\'ematiques d’Orsay, 91405, Orsay, France.}
\email{sebastian.munoz-thon@universite-paris-saclay.fr}

\author{Sean Richardson}
\address{Department of Mathematics, University of Washington, Seattle, WA 98195-4350.}
\email{seanhr@uw.edu}

\begin{document}

\begin{abstract}
Guillarmou's normal operator over a closed Anosov manifold is analogous to the classical normal operator of the geodesic X-ray transform over manifolds with boundary. In this paper, we generalize this normal operator, under some dynamical assumptions, to thermostat flows as well as to the case of the magnetic flows. In particular, we show that these generalized normal operators are elliptic pseudodifferential operators of order $-1$ in each case. As an application, we prove a stability estimate for the magnetic X-ray transform.
\end{abstract}

\maketitle

\section{Introduction}

Given a volume-preserving Anosov flow $\varphi_t \colon \mc{M} \to \mc{M}$ over a smooth closed manifold $\mc{M}$ with infinitesimal generator $X$, Guillarmou showed the existence of a self-adjoint bounded operator $\Pi \colon C^{\infty}(\mc{M}) \to \mathcal{D}'(\mc{M})$ from smooth functions to distributions \cite{Guillarmou17}. One of the main properties of this operator is that it generates distributions invariant under the flow, which has been an important topic of recent research in the case of geodesic flows $\varphi_t \colon SM \to SM$ over the unit tangent bundle $\pi_0 \colon SM \to M$ of a closed, oriented Riemannian manifold $M$ \cites{PSU14, 2015-paternain-salo-uhlmann-invariant}. In the case of geodesic flows, Guillarmou defined the \emph{normal operator} $\Pi_0 = \pi_{0*} \Pi \pi_0^*$ and proved that $\Pi_0$ is an elliptic pseudodifferential operator of order $-1$. This has applications to the X-ray transform $I_0 \defeq I \circ \pi_0^*$ where for each closed geodesic $\gamma$, we define $I$ to act on a smooth function $f \in C^{\infty}(SM)$ by
$$
    If(\gamma) = \int_0^{\operatorname{Length}(\gamma)} f(\varphi_t(x,v)) dt,
$$
and $(x,v)$ is a position-velocity pair of $\gamma$ at some time. By identifying a symmetric covariant $m$-tensor $h$ with a function on the unit tangent bundle by $\pi_m^* h (x,v) = h_x(v, \cdots, v)$, this X-ray transform generalizes to such $m$-tensors by $I_m \defeq I \circ \pi_m^*$. With applications to this X-ray transform $I_m$, Guillarmou generalized the previously introduced normal operator to symmetric covariant $m$-tensors by considering $\Pi_m = \pi_{m*} \Pi \pi_m^*$, and proved that $\Pi_m$ is an elliptic pseudodifferential operator of order $-1$. The operator $\Pi_m$ plays an analogous role to the normal operator for the X-ray transform on manifolds with boundary as studied in \cites{2004-stefanov-uhlmann, SSU 05}. Of particular importance is the case $m = 2$, for $I_2$ corresponds to the linearization of the marked length spectrum, which encodes the length of all closed orbits. Indeed, by using the normal operator $\Pi_2$, Guillarmou and Lefeuvre proved the marked length spectrum of an Anosov manifold locally determines the metric \cite{2019-guillarmou-lefeuvre} which was refined in \cite{GKL22}, making progress towards the Burns--Katok conjecture \cite{1985-burns-katok}.
    
In this paper, we generalize the theory of Guillarmou's normal operator from geodesic flows to the more general magnetic and thermostat flows. For a geodesic flow, a trajectory $\gamma(t)$ by definition does not accelerate with respect to the covariant derivative $D_t$ induced by the Levi-Civita connection: $D_t \dot{\gamma} = 0$. A magnetic flow $\varphi_t \colon SM \to SM$ has trajectories $\gamma(t)$ that model charged particles accelerated by a magnetic field, governed by $D_t \dot{\gamma} = Y(\gamma(t), \dot{\gamma}(t))$ where $Y \colon TM \to TM$ must be of the form $\ip{Y(x,v)}{w}_g = \Omega(v,w)$ for a closed 2-form $\Omega$ that encodes the magnetic field.

A thermostat flow $\varphi_t \colon SM \to SM$ generalizes this class of flows further and has trajectories $\gamma(t)$ governed by $D_t \dot{\gamma} = Y(\gamma(t), \dot{\gamma}(t))$ for any bundle map $Y \colon SM \to N$ taking values in the normal bundle $N(x,v) = \{w \in T_xM : \ip{v}{w}_g = 0\}$. We will give more preliminary details on thermostat and magnetic flows in Section~\ref{sec:thermostats}.

In Section~\ref{sec:thermo-normal}, we define an analogous normal operator $\Pi_0$ for the case of an Anosov thermostat flow. Then under certain dynamical assumptions (discussed in detail in Section \ref{sec:prelim}), we show this normal operator $\Pi_0$ is an elliptic pseudodifferential operator for thermostat flows.

\begin{thm}
	Consider a topologically transitive Anosov thermostat flow and corresponding normal operator $\Pi_0$. If the flow has no conjugate points and the stable and unstable bundles are each transverse to the vertical, then $\Pi_0$ is an elliptic pseudodifferential operator of order $-1$.
    \label{thm:main-thermostat-thm}
\end{thm}

This result has applications to the ray transform of an Anosov thermostat flow $\varphi_t: SM \to SM$ over functions given by $I_0 := I \circ \pi_0^*$ for
\[
    (I f)(\gamma) = \int_0^{T} f(\varphi_t(x,v)) dt
\]
where $\gamma$ is a closed thermostat orbit of period $T$ with position-velocity pair $(x,v)$ at some time. The injectivity of such ray transforms has been studied over closed surfaces \cite{DP07, JP09, AZ17}, but little is known for the higher dimensional case. By the same argument giving \cite[Theorem 16.2.6]{Lefeuvre}, Theorem~\ref{thm:main-thermostat-thm} immediately implies the kernel of the thermostat ray transform is finite dimensional for an Anosov thermostat flow over a manifold of any dimension.

Furthermore, we may relax the assumptions of Theorem~\ref{thm:main-thermostat-thm} for magnetic flows and thermostat flows over a surface. Indeed, it is known an Anosov flow will be free of conjugate points and that the stable and unstable bundles will be transverse to the vertical bundle for a magnetic flow \cite{1994-paternain-paternain} and in the case of a thermostat flow over a surface by \cite[Theorem 1.5]{ECMR25} and \cite[Lemma 4.1]{DP07}. Furthermore, we may remove the transitive assumption: magnetic flows are volume-preserving, so the Anosov property implies ergodicity (for $M$ connected), giving topological transitivity. For thermostats, Ghys proved any Anosov flow over the unit tangent bundle of a surface is topologically orbit equivalent to a geodesic flow over a surface of constant negative curvature \cite[Theorem A]{1984-ghys}, implying topological transitivity. Thus Theorem~\ref{thm:main-thermostat-thm} has the following immediate consequence.
\begin{cor}
    If $\varphi_t \colon SM \to SM$ is an Anosov flow over a closed connected manifold that is either
    \begin{enumerate}[(a), nosep]
        \item a magnetic flow,
        \item a thermostat flow over a surface,
    \end{enumerate}
    then the corresponding normal operator $\Pi_0$ is an elliptic pseudodifferential operator of order $-1$.
\end{cor}

In Section~\ref{sec:magnetic_normal} we generalize Guillarmou's normal operator in a different direction: to correspond with the magnetic ray transform over higher order tensors. In the case of the geodesic flow, the X-ray transform appears naturally as the linearization of the marked length spectrum. In the magnetic case, we have a similar phenomenom. Indeed, if $\Omega=d\alpha$ (i.e., for \emph{exact magnetic systems}), we define the \emph{magnetic action} over the curve $\gamma \colon [0,T] \to M$ by
\begin{equation} \label{eq:action}
    \mathbb{A}(\gamma)=\frac{1}{2}\int_{0}^{T}|\dot{\gamma}(t)|_{g}^{2}dt+\frac{1}{2}T-\int_{\gamma}\alpha.
\end{equation}
It was shown in \cite{CIPP00}*{Theorem 29}, that on every non-trivial free homotopy class, there is a minimizer of $\mathbb{A}$, i.e., a magnetic geodesic. In the case of Riemannian manifolds with boundary, the magnetic action was studied in \cite{DPSU07}. As in \cite{DPSU07}*{Lemma 3.1}, it can be shown that the linearization of $\A$ in the direction $[h,\beta] \in C^{\infty}(S^{2}(T^{*}M)) \times C^{\infty}(S^{1}(T^{*}M))$ is given by
\[ \int_{\gamma} \langle h,\gamma \otimes \gamma \rangle +\int_{\gamma} \langle \beta, \gamma\rangle .   \]
This motivates the definition of the magnetic ray transform by $I_m[p,q] \defeq I(\pi_m^* p + \pi^*_{m-1}q)$ for $[p,q] \in C^{\infty}(S^m(T^*M)) \times C^{\infty}(S^{m-1}(T^*M))$. In similar spirit to the normal operator considered in \cite{DPSU07} for manifolds with boundary, we define the \emph{magnetic normal operator}
\begin{equation} \label{eq:magnetic_normal}
    N_{m}=\begin{pmatrix}
    \pi_{m *} (\Pi+1\otimes 1) \pi_{m}^{*} & \pi_{m *} (\Pi+1\otimes 1) \pi_{m-1}^{*} \\ \pi_{m-1 *} (\Pi+1\otimes 1) \pi_{m}^{*}  & \pi_{m-1 *} (\Pi+1\otimes 1) \pi_{m-1}^{*}
\end{pmatrix},\
\end{equation}
which plays the role of the normal operator to $I_{m}$. This operator further generalizes Guillarmou's original normal operator for the geodesic flow and enjoys many of the same properties (which we explore in Section \ref{sec:magnetic_normal}), including its pseudodifferential nature:

\begin{thm} \label{thm:normal_magnetic_psido}
$N_{m} \in \Psi_{\mathrm{cl}}^{-1}(M)$ with principal symbol
\[
\begin{split}
    \sigma_{N_{m}}(x,\xi)=\frac{2\pi}{|\xi|} \mathrm{diag} \bigg[  &C_{n,m}^{-1}\pi_{\ker_{i_{\xi}^{m}}}\pi_{m *}\pi_{m}^{*}\pi_{\ker_{i_{\xi}^{m}}}, \\
    & C_{n,m-1}^{-1}\pi_{\ker_{i_{\xi}^{m-1}}}\pi_{m-1 *}\pi_{m-1}^{*}\pi_{\ker_{i_{\xi}^{m-1}}} \bigg].
\end{split}
\]
\end{thm}

The expression associated with the symbol follows from a potential-solenoidal decomposition associated to a magnetic differential operator that we study in Section \ref{sec:mag_ps}. The super-indices over $i_{\xi}$ denote the order of tensors on which it acts. Here, $\pi_{\ker_{i_{\xi}}}$ is the projection onto $\ker i_{\xi}$, where $i_{\xi}$ is the contraction operator. Similarly as in the case of the geodesic flow, we have $C_{n,m}=\sqrt{\pi} \frac{\Gamma(\frac{n-1}{2}+m)}{\Gamma(\frac{n}{2}+m)}$.

As a consequence of the pseudodifferential nature and Liv\v sic theory developed by Gou\"ezel and Lefeuvre, we obtain a stability estimate for $I_{m}$.

\begin{thm} \label{thm:stability_xray}
    Assume that $I_{m}$ is s-injective. There exists $s_0 \in(0,1)$ and $C, \tau>0$ such that for all $f \in C^{1}(M, S^{m}(T^{*}M) \times S^{m-1}(T^{*}M))$:
    \[ \|\pi_{\ker D_{\mu}^{*}} f \|_{H^{s_0}} \leq C \|I_m f \|_{\ell^{\infty}(\mathcal{C})}^{\tau}\|f\|_{C^1}^{1-\tau}.\]
\end{thm}

Here s-injectivity means that the kernel of $I_{m}$ is as big as possible: it is exactly the range of a differential operator $D_{\mu}$ which is closely related with the generator of the magnetic flow $F$. Indeed, we have that $F[\pi_{m}^{*},\pi_{m-1}^{*}]=[\pi_{m+1}^{*},\pi_{m}]D_{\mu}$ (this is the content of Lemma \ref{lemma:commutating}). Theorem \ref{thm:stability_xray} is proved in Section \ref{sec:inj_stab_xray}. There, we also adapt the proof of injectivity for $I_{m}$ given in \cite{DPSU07} for manifolds with boundary, under the assumptions that $g$ is negatively curved and the magnetic field is small. This shows in particular that Theorem \ref{thm:stability_xray} is a non-void result.

Finally, let us comment that, besides the motivation of generalizing the $\Pi_{m}$ operator, the ``more natural'' motivation (and also for what the original $\Pi_{m}$ was created) was to study rigidity questions. As we just mentioned above, $N_{m}$ is related with $I_{m}$ which is the linearization of the magnetic action. Following the Burns--Katok conjecture, one could ask if the marked magnetic action determines the magnetic system $(g,\alpha)$ up to a proper group of transformations (or \emph{gauge}). In this case, besides the usual diffeomorphism, there is a exact one form on the magnetic part. More explicitly, if $\phi$ is a diffeomorphism homotopic to the identity and $\psi$ and smooth function, then the marked magnetic action spectrum (i.e., the set of magnetic actions of magnetic geodesics on each free homotopy class) of $\A_{g,\alpha}$ and $\A_{\phi^{*}g,\phi^{*}\alpha+d\psi}$ coincide. A similar problem was studied on manifolds with boundary in \cite{DPSU07}. On closed surfaces, another generalization of the marked length spectrum was recently studied in \cite{AdSMRT24}. They obtained rigidity results by using the (Riemannian) length of magnetic geodesics.

\subsection*{Acknowledgements} The authors thank T. Lefeuvre and G. Paternain for useful discussions, and J. Echevarría Cuesta and J. Marshall Reber for discussions over conjugate points on surfaces. This paper was partially written while the second author was visiting Stanford University and University of Washington, he gratefully acknowledges the hospitality and support of both institutions during his stay. S.M-T. was supported by the European Research Council (ERC) under the European Union’s Horizon 2020 research and innovation programme (Grant agreement no. 101162990 -- ADG) and S.R. was supported by the National Science Foundation Graduate Research Fellowship under Grant No. DGE-2140004.

\section{Preliminaries} \label{sec:prelim}

\subsection{Unit tangent bundle and geodesic flow}
\label{sec:SM}

For any Riemannian manifold $(M,g)$, the \emph{unit tangent bundle} is defined $SM = \{(x,v) \in TM : |v|_g = 1\}$ and comes with natural projection $\pi_0 \colon SM \to M$. The unit tangent bundle is the phase space for several relevant flows that we introduce in this section. As a first example, consider the \emph{geodesic flow} $\varphi_t \colon SM \to SM$ given by $\varphi_t(x,v) = (\gamma_{x,v}(t), \dot{\gamma}_{x,v}(t))$ where $\gamma_{x,v}(t)$ is the geodesic with initial position $\gamma_{x,v}(0) = x$ and initial velocity $\dot{\gamma}_{x,v}(0) = v$. We denote the  infinitesimal generator of the geodesic flow by $X$, a smooth vector field on $SM$.

The unit tangent bundle $SM$ inherits natural geometric structures from the base manifold $M$, which we now recall. Smooth curves $Z \colon (-\eps, \eps) \to SM$ over the unit tangent bundle take the form $Z(t) = (\alpha(t), W(t))$ for a smooth curve $\alpha \colon (-\eps, \eps) \to M$ over the base manifold together with a smooth vector field $W(t)$ along $\alpha(t)$. Thus we may think of a tangent vector $\xi \in T_{(x,v)}SM$ as the velocity vector of such a curve $\xi = \derev{}{t}{t=0}(\alpha(t), W(t))$ so that $\alpha(0) = x$ and $W(0) = v$. This expression together with the Levi-Civita connection $\nabla$ and corresponding covariant derivative $D_t$ along $\alpha(t)$ allows us to naturally identify $\xi$ with $\left(\alpha'(0), \left.D_t W(t)\right|_{t=0}\right)$, which we can write $(d\pi_0 \xi, \mathtt{K} \xi)$ using the \emph{connection map} $\mathtt{K}\xi \defeq \left.D_tW\right|_{t=0}$. To formalize this splitting, consider subbundles $\V \defeq \ker d\pi_0$ and $\widetilde{\mathbb{H}} \defeq \ker \mathtt{K}$, then $d\pi_0 \colon \widetilde{\mathbb{H}}(x,v) \to T_xM$ and $\mathtt{K} \colon \mathbb{V}(x,v) \to \{v\}^{\perp}$ are linear isomorphisms, which allows us to write $TSM = \widetilde{\mathbb{H}} \oplus \V$. To represent an element $\xi \in T_{(x,v)}SM$ under this splitting, as before, we will write $\xi = (d\pi_0(\xi), \mathtt{K} \xi)$. Declaring this splitting to be orthogonal and using the metric induced by $g$ on each subbundle gives the \emph{Sasaki metric} $G$ on $SM$ defined by
\begin{equation*}
    \ip{\xi}{\eta}_G \defeq \ip{d\pi (\xi)}{d\pi (\eta)}_g + \ip{\mathtt{K}\xi}{\mathtt{K}\eta}_{g}
\end{equation*}
for $\xi, \eta \in T_{(x,v)}SM$. The Sasaki metric induces the \emph{Liouville volume form} $d\Sigma$ as well as volume form $dS_x$ on $S_xM$. We use the Sasaki metric to orthogonally decompose $\widetilde{\H} = \R X \oplus \H$ further. In summary, we arrive at the splitting
$$TSM = \X \oplus \H \oplus \V,$$
into the \emph{flow direction} $\X \defeq \R X$, the \emph{horizontal subbundle} $\H$, and the \emph{vertical subbundle} $\V$. See \cites{Paternain, PSU23} for more details.
% If $M$ is a surface, these subbundles are spanned by vector fields $X$, $H$, and $V$ respectively, normalized to satisfy the \emph{structure equations} $[X,V] = H$, $[H,V] = -X$, and $[X,H] = -KV$ where $K$ is the Gaussian curvature. 

% Taking $\{X,H,V\}$ to be an orthonormal frame defines the \emph{Sasaki metric} on $SM$, which induces induces volume form $d\Sigma$ on $SM$ and volume form $dS_x$ on $S_xM$ -- 

\subsection{Magnetic and thermostat flows}
\label{sec:thermostats}
Let $(M,g)$ be a closed Riemannian manifold with Levi-Civita connection $\nabla$. A magnetic field over $M$ may be encoded by a closed 2-form $\Omega$ over $M$, or by the resulting map $Y \colon TM \to TM$ given by $\ip{Y(x,v)}{w}_g = \Omega(v,w)$, which governs the acceleration of a particle with unit mass and charge by the \emph{Lorentz--Newton law}
\begin{equation}
    D_t\dot{\gamma} = Y(\gamma(t),\dot{\gamma}(t)).
    \label{eq:lorentz}
\end{equation}
Here $D_t$ denotes the covariant derivative along $\gamma(t)$ induced by the Levi-Civita connection. Let $\gamma_{x,v}(t)$ denote the unique solution to \eqref{eq:lorentz} with initial conditions $\gamma(0) = x$ and $\dot{\gamma}(0) = v$. These dynamics induce the \emph{magnetic flow} $\varphi_t \colon SM \to SM$ on the unit tangent bundle by $\varphi_t(x,v) = (\gamma_{x,v}(t), \dot{\gamma}_{x,v}(t)$) with infinitesimal generator $F(x,v) = \derev{}{t}{t=0}(\gamma_{x,v}(t), \dot{\gamma}_{x,v}(t))$. Under the splitting $TSM = \widetilde{\H} \oplus \V$ described in Section~\ref{sec:SM}, we may identify $F(x,v)$ with $(\dot{\gamma}_{x,v}(0), D_t\dot{\gamma}_{x,v}(0))$, which is simply $(v, Y(x,v))$ by the Lorentz force law \eqref{eq:lorentz}.
We will often use this splitting implicitly, for example by writing $F(x,v) = (v, Y(x,v))$.

In addition to magnetic flows, we will work with the more general \emph{thermostat flows} $\varphi_t\colon SM \to SM$ given by infinitesimal generator $F(x,v) = (v, Y(x,v))$ under the splitting $TSM = \widetilde{\H} \oplus \V$ where the force $Y \colon SM \to N$ is a general bundle map taking values in the normal bundle $N(x,v) = \{w \in T_xM : \ip{v}{w}_g = 0\}$.

Note $F$ has nonzero component in the direction $X$ of the geodesic flow, so we may write the splitting $TSM = \FF \oplus \H \oplus \V$ where we are denoting $\FF \defeq \R F$.

\subsection{Anosov Flows}
\label{sec:anosov}
The theory in defining Guillarmou's operator $\Pi$ requires the dynamics of the thermostat flow to be sufficiently chaotic and, in particular, we require the Anosov property. Recall a flow $\varphi_t \colon \mc{M} \to \mc{M}$ over a closed smooth manifold $\mc{M}$ with infinitesimal generator $F$ is called \emph{Anosov} if there is a continuous and flow-invariant splitting of the tangent space
\begin{equation*}
	T\mc{M} = \FF \oplus E_s \oplus E_u,
\end{equation*}
into the flow direction $\FF := \R F$, the \emph{stable bundle}, and the \emph{unstable bundle} respectively such that given an arbitrary metric $|\bullet|$ on $\mc{M}$ (recall that $\mc{M}$ is closed) there exists constants $C, \lambda > 0$ so that for all $t \geq 0$
\begin{align*}
	|d\varphi_t(v)| &\leq Ce^{-\lambda t}|v| \quad \text{for } v \in E_s,\\
	|d\varphi_{-t}(v)| &\leq Ce^{-\lambda t}|v| \quad \text{for } v \in E_u.
\end{align*}
For example, all Riemannian manifolds with strictly negative sectional curvature have Anosov geodesic flow \cite{1967-anosov}. 
When applying microlocal tools, we are often interested in the dual splitting
\begin{equation*}
	T^*\mc{M} = E^*_0 \oplus E^*_s \oplus E^*_u ,
\end{equation*}
defined so that $E^*_0(E_s \oplus E_u) = 0$, $E^*_s(E_s \oplus \FF) = 0$, and $E^*_u(E_u \oplus \FF) = 0$.

\subsection{Liv\v sic theory}

In dynamics, Liv\v sic theory studies the obstructions to solve $Fu=f$, where $F$ is the generator of a certain Anosov flow $\varphi_{t} \colon \mathcal{M} \to \mathcal{M}$. In our context, this translates into nice properties for functions in (or near to) the kernel of X-ray transforms. 
The classical result is the following.

\begin{lemma} \label{lemma:livsic}
    Let $\varphi_{t} \colon \mathcal{M} \to \mathcal{M}$ be a transitive Anosov flow. If $f \in C^{\infty}(\mathcal{M})$ integrates to zero on any closed trajectory of $\varphi$, then there exists $u \in C^{\infty}(\mathcal{M})$ such that $f=Fu$.
\end{lemma}

The original theorem by Liv\v sic himself is in H\"older regularity, but we use the smooth version from \cite{LlMM_smooth_livsic}. We furthermore use some variations of the classical version. A second result of interest for us is the non-negative Liv\v sic Theorem \cite{LT05}, which deals with functions whose integrals are, as the name suggests, non-negative.

\begin{lemma} \label{lemma:nnl}
    Let $\varphi_{t} \colon \mathcal{M} \to \mathcal{M}$ be a transitive Anosov flow. If the integral of $f$ over any closed orbit of the flow is non-negative, then there exists $u \in C^{\beta}(\mathcal{M})$ such that for any $x$ and $T$ we have
    \[
    \int_{0}^{T} f(\varphi_{t}x)dt \geq u(\varphi_{T}(x))-u(x).
    \]
\end{lemma}

Another crucial result is the Approximate Liv\v sic Theorem, originally proved in \cite{GL21} (see also \cite{Lefeuvre}*{Theorem 11.1.5}). Here, one deals with functions that integrate to at most $\eps$.

\begin{lemma} \label{lemma:approx}
    Let $\varphi_{t} \colon \mathcal{M} \to \mathcal{M}$ be a transitive Anosov flow. There exists $C,\tau,\beta>0$ constants such that the following holds: for any $f \in C^{1}(\mathcal{M})$ with 
    \[ \sup_{\gamma \in \mathcal{P}}|If(\gamma)|<\eps, \]
    there exist $u,h \in C^{\beta}(\mathcal{M})$ such that $f=Fu+h$ and
    \[ \| h\|_{C^{\beta}(\mathcal{M})} \leq C\eps^{\tau}\|f\|_{C^{1}(\mathcal{M})}^{1-\tau}. \]
\end{lemma}

Here $\mathcal{P}$ denoted the set of periodic orbits of $\varphi_{t}$.

\subsection{Conjugate points and transversality of the vertical}
\label{sec:conjugate-points}

An Anosov geodesic flow guarantees two useful properties: the absence of conjugate points and the transversality of the vertical bundle to the stable and unstable bundles \cite{Paternain}. Both properties are key to proving the ellipticity of Guillarmou's $\Pi$ operator in \cite{Guillarmou17}, so we review these properties and their generalizations.

Given a flow $\varphi_t \colon SM \to SM$ over the unit tangent bundle of a smooth manifold, we call two points $(x,v)$ and $\varphi_t(x,v)$ \emph{conjugate} if the vertical intersects itself non-trivially under the flow: $d\varphi_t(\V(x,v)) \cap \V(d\varphi_t(x,v)) \neq \{0\}$. For a geodesic flow over a Riemannian manifold, this is equivalent to the usual definition of conjugate points with Jacobi fields.

Using the splitting $TSM = \FF \oplus \H \oplus \V$ described in Section~\ref{sec:thermostats}, we have an equivalent dual description of conjugate points for thermostat flows. We say that the points $(x,v)$ and $\varphi_t(x,v)$ are \emph{conjugate} exactly when there exists $t$ such that
$$d\varphi_t^\top(\FF^* \oplus \H^*(x,v)) \cap \FF^* \oplus \H^*(\varphi_t(x,v)) \neq \{0\}.$$
Let $\Phi \colon T^*SM \to T^*SM$ denote the Hamiltonian lift $\Phi_t(x,\xi) = (\varphi_t(x), d\varphi_t^\top \xi)$, which allows us to once again rephrase the characterization of conjugate points: the point $(x,\xi) \in T^*SM$ is conjugate to $\Phi_t(x,\xi)$ if and only if $(x,\xi) \in \FF^* \oplus \H^*$ implies $\Phi_t(x,\xi) \notin \FF^* \oplus \H^*$.

Next, we recall that for a geodesic Anosov flow, we have $E_s \cap \V = \{0\} = E_u \cap \V$ where $E_s$ and $E_u$ denote the stable and unstable bundles respectively and $\V$ denotes the vertical bundle \cites{1974-klingenberg-riemannian, 1987-mane-on-a-theorem-of-klingenberg}. 

For a general thermostat flow, it is unknown if the Anosov property guarantees the absence of conjugate points or if the vertical bundle is transverse to the stable and unstable bundles. Therefore, we will take these properties as an assumption when necessary. However, the absence of conjugate points and this transversality property is known to hold in both the case of magnetic flows \cite{1994-paternain-paternain} and in the case of thermostats over a surface \cite{ECMR25} and therefore both assumptions may be removed for these particular cases.

\subsection{Anisotropic Sobolev Spaces}

In this subsection we recall the construction and the main properties of certain Hilbert spaces that will be useful to our purposes. The results presented here are due to Faure--Sj\"ostrand \cite{2011-faure-sjostrand}, we also refer to \cite{Lefeuvre}*{Chapter 9}.

Let $\varphi_{t} \colon \mathcal{M} \to \mathcal{M}$ be an Anosov flow with generator $F$. Let $H$ be the Hamiltonian vector field on $T^{*}\mathcal{M}$ induced by the Hamiltonian $\sigma_{P}(x,\xi):=\langle\xi, F(x)\rangle$, the principal symbol of $P:=\frac{1}{i} F$, and let $(\Phi_{t})_{t \in \mathbb{R}}$ be the symplectic flow generated by $H$. Fix an \emph{order function} $m \in C^{\infty}(T^{*}\mathcal{M},[-1,1])$, i.e., $m$ is such that is 0-homogeneous in the $\xi$-variable for $|\xi|$ large enough, such that $m \equiv 1$ in a conic neighborhood of $E_{s}^{*}$, and $m \equiv -1$ in a conic neighborhood of $E_{u}^{*}$. Consider the associated \emph{escape function} $G_{m}(x,\xi):=m(x,\xi) \log|\xi|_{g}$. Then:
\begin{enumerate}[label=\roman*)]
    \item $H G_{m}(x, \xi) \leq 0$  for $|\xi|_{g} \geq R$;
    \item $H G_{m}(x, \xi) \leq-C<0$ for all $(x,\xi)$ in a conic neighborhood of $E_{s}^{*}\cup E_{u}^{*}$ and such that $|\xi|_{g} \geq R$, for some constant $C>0$.
\end{enumerate}
We now define $A_{s}:=\mathrm{Op} (e^{s G_{m}})$. Next, we define the \emph{anisotropic Sobolev spaces} by 
\[
\mathcal{H}_{+}^s(\mathcal{M}):=A_{s}^{-1}(L^2(\mathcal{M})), \quad \mathcal{H}_{-}^s(\mathcal{M}):=A_{s}(L^2(\mathcal{M}))
\]
We have that $C^{\infty}(\mathcal{M}) \subset \mathcal{H}_{h,\pm}^{s}(\mathcal{M}) \subset \mathcal{D}'(\mathcal{M})$. Furthermore:

\begin{lemma} \label{lemma:asino_prop} 
For any $s \geq 0$, the following inclusions are continuous:
\[ 
H^{s}(\mathcal{M}) \subset \mathcal{H}_{\pm}^{s}(\mathcal{M}) \subset H^{-s}(\mathcal{M})
\]
\end{lemma}

The moral of these spaces is the following: the space $\mathcal{H}_{+}^{s}$ (resp. $\mathcal{H}_{-}^{s}$) is defined microlocally to have Sobolev regularity $H^{s}$ (resp. $H^{-s}$) in a conic neighborhood of $E_{s}^{*}$ and regularity $H^{-s}$ (resp. $H^{s}$) in a conic neighborhood of $E_u^*$. 

\subsection{Meromorphic extension of resolvents}
\label{sec:resolvents}

Let $\varphi_{t} \colon \mathcal{M} \to \mathcal{M}$ a volume preserving Anosov flow, and let $F$ be its generator. Then $iF$ is an unbounded self-adjoint operator in $L^{2}(\mathcal{M})$, with $L^{2}$-spectrum contained in $\R$ and the resolvents 
\[
R_{\pm}(z) = (\mp F-z)^{-1}:=-\int_{0}^{\infty}e^{\mp tF}e^{-tz}dt.
\]
are well defined bounded operators on $L^{2}(\mathcal{M})$ for $\real (z)>0$ big enough. However, we can extend it when we consider it as an operator between anisotropic Sobolev spaces.

\begin{lemma} [\cite{2011-faure-sjostrand}, \cite{Lefeuvre}*{Theorem 9.1.6}] \label{lemma:resolvent_aniso}
   Let $\varphi_{t} \colon \mathcal{M} \to \mathcal{M}$ an Anosov flow. Then,
   \[
   R_{\pm}(z)=(\mp F-z)^{-1} \colon \mathcal{H}_{\pm}^{s}(\mathcal{M}) \to \mathcal{H}_{\pm}^{s}(\mathcal{M})
   \]
    is a meromorphic family of bounded on $\{\real(z)>-cs+\mu\}$, where $c,\mu>0$ are constants depending on the flow, and $s>0$. Furthermore, the poles do no depend on the choices made on the construction of $\mathcal{H}_{\pm}^{s}(\mathcal{M})$.
\end{lemma}

Thus around any pole $z_0 \in \C$, the resolvents admit a Laurent expansion
$$
    R_{\pm}(z) = R^{\hol}_{\pm}(z) - \sum_{k=1}^{N(z_0)} \frac{(\mp F-z)^{k-1}\Pi_{z_0}^{\pm}}{(z-z_0)^k}.
$$
where the operators $\Pi^{\pm}_{z_0}$ are the spectral projections. 

The normal operator we consider in this paper is built from the holomorphic part $R_{\pm}^{\hol}(z)$ of each resolvent and we will need the following characterization of its wavefront set.

\begin{lemma}[\cite{2016-dyatlov-zworski}]
		\begin{equation*}
			\WF(R^{\hol}_+(z)) \subset \Delta(T^*\mathcal{M}) \cup \Omega_+ \cup (E_u^* \times E_s^*),
		\end{equation*}
		where $\Delta(T^* \mathcal{M})$ is the diagonal, and
		\begin{equation*}
			\Omega_+ = \{(\Phi_t(x,\xi), x, \xi) : t \geq 0 \text{ $\mathrm{ and }$ } \xi(F(x)) = 0 \},
		\end{equation*}
		is the positive flow-out.
		\label{lem:res-WF}
	\end{lemma}

	Recall for an operator $A \colon C^{\infty}(Y) \to C^{\infty}(X)$ the wavefront $\WF(A) \in T^*(X \times Y)$ is defined to be the wavefront of the Schwartz kernel $K_A \in \mathcal{D}'(X \times Y)$ of $A$. Importantly, if $u \in \mathcal{D}'(Y)$ then
	\begin{equation*}
		\WF(Au) \subset \WF'(A) \circ \WF(u) \cup \WF_X'(A),
	\end{equation*}
	where we are using the notations
	\begin{align*}
		\WF_X'(A) &= \{(x, \xi) \in T^*X \sm 0 : \exists y \in Y \text{ s.t. } (x,\xi, y, 0) \in \WF'(A)\},\\
		\WF'(A) \circ \WF(u) &= \{(x,\xi) \in T^*X \sm 0 : \exists (y,\eta) \in \WF(u) \text{ s.t. } (x,\xi,y,\eta) \in \WF'(A)\},
	\end{align*}
	and in the above
	\begin{equation*}
		\WF'(A) \defeq \{(x,\xi,y,-\eta) : (x,\xi,y,\eta) \in \WF(A)\}.
	\end{equation*}
    Then the expression for $\WF(R^{\hol}_{\pm})$ given in Lemma~\ref{lem:res-WF} bounds $\WF(R_+(z)u)$ by
	
	\begin{cor}
		Take $u \in \mathcal{D}'(M)$, then
		\begin{align*}
			&\WF(R^{\hol}_+(z)u)\\
			&\subset \{(x,-\xi): (x, \xi) \in \WF(u)\}\\ %from \Delta T^*M term
			&\cup \{(x,\xi) \in T^*\mathcal{M} \sm 0 : \exists (y, -\eta) \in \WF(u) \text{ $\mathrm{s.t.}$ } \eta(F(x)) = 0, (x,\xi) = \Phi_t(y,\eta)\}\\ %from \Omega_+ term\\
			&\cup \{(x,\xi) \in T^*\mathcal{M} \sm 0 : \exists (y,-\eta) \in \WF(u) \text{ $\mathrm{s.t.}$ } (x,\xi,y,-\eta) \in E_u^* \times E_s^*\}.
		\end{align*}
		\label{lem:WF-of-Ru}
	\end{cor}

Mixing Anosov flows preserving a smooth volume form have a simple pole at $0$ \cite{Guillarmou17}*{Lemma 2.5} and so in this case
\begin{equation}
    R_{\pm}(z) = R^{\hol}_{\pm}(z) - \frac{\Pi^{\pm}_{0}}{z}.
    \label{eq:simple-pole}
\end{equation}

For a topologically transitive Anosov flow on a 3-manifold (e.g. $SM$ where $M$ is a surface), then \eqref{eq:simple-pole} also holds by \cite[Lemma 3.1]{2022-cekic-paternain}. In fact, the same argument by Cecki\'c and Paternain holds verbatim for arbitrary dimension, so \eqref{eq:simple-pole} holds for a general topologically transitive Anosov flow.

\subsection{The Normal Operator}

For a topologically transitive Anosov flow, we consider \emph{Guillarmou's operator} $\Pi :=-(R_{+}^{\hol}(0) + R_{-}^{\hol}(0))$. 

Next, suppose the Anosov flow $\varphi_t \colon SM \to SM$ is over the unit tangent bundle of a Riemannian manifold $M$. Recall a symmetric covariant $m$-tensor $h \in C^{\infty}(S^m(T^*M))$ can be naturally identified with a function over $SM$ by the map $\pi_m^* \colon C^{\infty}(S^m(T^*M)) \to C^{\infty}(SM)$ defined by $(\pi^*h)(x,v) = h_x(v, \ldots, v)$. This identification map has natural adjoint 
\begin{align*}
\pi_{m*} \colon \mathcal{D}'(SM) \to \mathcal{D}'(S^m(T^*M)), \quad
\ip{\pi_{m*}f}{h}_{S^m(T^*M)} \defeq \ip{f}{\pi_{m}^*h}_{SM}
\end{align*}
given by the natural pairing with a smooth symmetric covariant $m$-tensor $h \in C^{\infty}(S^m(T^*M))$ induced by the metric on $M$. We also recall (\cite{Lefeuvre}*{Lemma 14.1.6}) that 
\begin{equation} \label{eq:pi_m^*}
    (\pi_{m}^{*}h)(v_{1},\ldots,v_{m})=\int_{S_{x}M}h(v)g(v,v_{1}) \cdots g(v,v_{m})dv.
\end{equation}
Considering these three operators together gives the \emph{normal operator} $\Pi_m \defeq \pi_{m*}(\Pi + 1 \otimes 1)\pi_m^*$. 
%%%%
We also recall some properties that will be useful throughout this work. They were originally proved in \cite{Guillarmou17}, see also \cite{Lefeuvre}*{Lemma 9.2.9}.

\begin{lemma} \label{lemma:pi_properties}
Let $\varphi_{t} \colon SM \to SM$ be an topologically transitive Anosov flow with generator given by $F$. Then:
    \begin{itemize}
        \item[(i)] $\Pi F=0$.
    \end{itemize}
    If in addition the flow preserves a volume form, then:
    \begin{itemize}
        \item[(ii)] $\Pi$ is formally self-adjoint.
        \item[(iii)] $\langle \Pi f,f \rangle_{L^{2}}  \geq 0$. Moreover, the equality is equivalent to $\Pi f=0$, and this is furthermore equivalent to the existence of $v \in \ker F$ and $u \in C^{\infty}(SM)$ so that $f=Fu+v$.
    \end{itemize}
\end{lemma}

\section{Normal operator for Thermostat Flows}
\label{sec:thermo-normal}

We may directly extend the definition of the normal operator from geodesic flows to thermostat flows. Indeed, consider a general Anosov thermostat flow $\varphi_t \colon SM \to SM$ as defined above (Section \ref{sec:thermostats}) and let $R_{\pm}(z) = (\mp X-z)^{-1}$ denote the meromorphic extension of the resolvents defined in Section~\ref{sec:resolvents}. Then we may define the operator $\Pi \defeq -(R_{+}^{\hol}(0) + R_{-}^{\hol}(0))$ in exactly the same way, which allows us to define the normal operator $\Pi_m \defeq \pi_{m*}(\Pi + 1 \otimes 1)\pi_{m}^*$. In this section, we prove Theorem~\ref{thm:main-thermostat-thm}. To begin with it, we write the holomorphic parts as follows:

    \begin{lemma}
    \label{lem:smoothing1}
        Suppose the resolvent $R_{\pm}(z)$ of an Anosov flow $\varphi_t \colon \mathcal{M} \to \mathcal{M}$ has a simple pole at $0$. Then we have
        \begin{equation}
            R^{\hol}_{\pm}(0) = \int_0^{\infty} \chi(t) \varphi^*_{\mp t}dt + R_{\pm}^{\hol}(0)\int_0^{\infty} \chi'(t) \varphi^*_{\mp t}dt - \int_0^{\infty} \chi'(t) dt \Pi_0^{\pm}
            \label{eq:hol-res-expression}
        \end{equation}
        where $\chi(t)$ is any compactly supported smooth bump function that is $1$ near $0$.
        \label{lem:resolvent-comp}
    \end{lemma}

    \begin{proof}

        We give the computation for $R_+^{\hol}(0)$ following \cite{Guillarmou17, Lefeuvre} with the case of $R_-^{\hol}(0)$ being similar. To begin, take any smooth bump function $\chi(t)$ so that $\chi \equiv 1$ on $[0, \eps/2]$ and $\chi \equiv 0$ on $[\eps, \infty)$, then for $\real(z) > 0$ decompose the resolvent as follows.
	\begin{align}
		R_+(z) &= \int_0^{\infty} e^{-tz}\varphi^*_{-t}dt\\
		&= \int_0^{\infty} \chi(t)e^{-tz}\varphi^*_{-t}dt
		+ \int_0^{\infty} (1-\chi(t))e^{-tz}\varphi^*_{-t}dt\\
		&= \int_0^{\infty} \chi(t) e^{-tz} \varphi^*_{-t}dt
		+ R_+(z)\int_0^{\infty} \chi'(t) e^{-tz}\varphi^*_{-t}dt.
        \label{eq:res-exp}
	\end{align}

    Now we are equipped to compute $R^{\hol}_+(0) = \lim_{z \to 0} R^{\hol}_+(z)$. Taking the limit of the holomorphic part of the above and using $e^{-tz} = 1 - tz + \mc{O}((tz)^2)$ yields
    \begin{align*}
		R_+^{\hol}(0) &= \lim_{z \to 0}R_+^{\hol}(z)\\
		&= \int_0^{\infty} \chi(t) \varphi^*_{-t}dt + R_+^{\hol}(0)\int_0^{\infty} \chi'(t) \varphi^*_{-t}dt - \int_0^{\infty} \chi'(t) dt \Pi_0^+.
	\end{align*}
    Note the last term in the above expression follows by the integration by parts and the following computation
    \begin{align*}
        \Pi_{0}^{+}\int_{0}^{\infty}\chi'te^{-tF}dt&=\Pi_{0}^{+} \left( \chi te^{-tF}|_{0}^{\infty}-\int_{0}^{\infty}\chi e^{-tF}dt-\int_{0}^{\infty}\chi tFe^{-tF}dt\right) \\
        &=-\int_{0}^{\infty}\chi \Pi_{0}^{+}e^{-tF}dt-\int_{0}^{\infty}\chi t(\Pi_{0}^{+}F)e^{-tF}dt
        =-\int_{0}^{\infty}\chi dt \Pi_{0}^{+},
    \end{align*}
    where in the last equality we used that for the first integral that $\Pi_{0}^{+}e^{-Ft}=\Pi_{0}^{+}$ since the $\Pi_{0}^{+}$ is the projection onto the set of invariant functions, and for the second one we used the fact $\Pi_{0}^{+}F=0$.

    \end{proof}
	
    The decomposition for the resolvents provided in Lemma~\ref{lem:smoothing1} yields a decomposition for the normal operator $\Pi_m = \pi_{m*}\Pi\pi_m^*$ by composing with the pullback operator $\pi_m^*$ and the pushforward operator $\pi_{m*}$. After applying these operators, we will follow \cite{Guillarmou17} to show the middle term of \eqref{eq:hol-res-expression} is smoothing. First, however, we must recall the wavefront set of the operators $\pi_m^*$ and $\pi_{m*}$. Indeed, recall the wavefront of the pullback operator is constrained as follows.
    %TODO cite proposition numbers below...
    \begin{lemma}[\cite{Lefeuvre}]
        \begin{equation}
           \WF(\pi_m^* h) \subset \{(x, \xi) \in T^*M : \exists (y,\eta) \in \WF(h) \text{ $\mathrm{s.t.}$ }  y = \pi_0(x), \xi = d\pi_0^\top \eta\}.
           \label{eq:pullback-wf}
        \end{equation}
    \end{lemma}

    Next, note the wavefront of $\pi_m^*$ selects only the wavefront in $\X \oplus \H$ as follows.
    \begin{lemma}[\cite{Lefeuvre}]
        \begin{equation}
            \WF(\pi_{m*} u) \subset \{(x,\xi) : \exists v \in S_xM \text{ $\mathrm{s.t.}$ } ((x,v), d\pi^\top \xi) \in \WF(u)\}.
            \label{eq:pushforward-wf}
        \end{equation}
    \end{lemma}

    Additionally, the following result will allow us to further restrict the wavefront of terms in \eqref{eq:hol-res-expression}.
    \begin{lemma}[\cite{Lefeuvre}]
	Let $F$ generate the flow $\varphi_t \colon \mc{M} \to \mc{M}$ with symplectic lift $\Phi \colon T^*\mc{M} \to T^*\mc{M}$ and let $\chi \in C_c^{\infty}(\R)$ be a cutoff function. Then,
	\begin{align*}
		&\WF\l(\int_{-\infty}^{\infty}\chi(t)\varphi_t^* dt u\r)\\
		&\subset \{(x,\xi) \in T^*M : \xi(F(x)) = 0 \text{ $\mathrm{and}$ } \exists t \in \supp(\chi) \text{ $\mathrm{s.t.}$ } \Phi_t(x,\xi) \in \WF(u)\}.
	\end{align*}
	\label{lem:integral-WF}
	\end{lemma}

    The preceding descriptions of the wavefront set allow us to conclude

	\begin{lemma} \label{lemma:smoothing_middle} Let $\varphi_t \colon SM \to SM$ be a flow such that
    \begin{enumerate}[(i)]
        \item the flow $\varphi_t$ is topologically transitive,

        \item the stable (resp. unstable) bundle of the flow is transverse to the vertical bundle,
        \item the flow $\varphi_t$ has no conjugate points.
    \end{enumerate}
    Then letting $R \coloneq R_{+}^{\hol}(0)$ (resp. $R \coloneq R_{-}^{\hol}(0)$), the following operator is smoothing:
	\begin{equation}
        \pi_{m*}R\int_0^{\infty} \chi'(t) \varphi^*_{-t} \pi_m^*.
        \label{eq:smoothing-op}
    \end{equation}
        \label{lem:smoothing-term}
	\end{lemma}

    \begin{proof}

	We successfully apply each of the composed operators in \eqref{eq:smoothing-op} while keeping track of the wavefront set as done in \cite{Guillarmou17} for geodesic flows. To begin, let $u \in \mathcal{D}'(M, S^mT^*M)$ be a symmetric $m$-tensor distribution, and apply the leftmost operator $\pi^*_m \colon \mathcal{D}'(M, S^mT^*M) \to \mc{D}'(SM)$, recalling the wavefront set of the resulting distribution $\pi^*_m u \in C^{\infty}(SM)$ is controlled by \eqref{eq:pullback-wf}.

	In particular, since $\ker(d\pi_0) = \V$, we have $d\pi_0^\top \eta \in \FF^* \oplus \H^*$. So
	\begin{equation}
		\WF(\pi_m^* u) \subset \FF^* \oplus \H^*.
        \label{eq:wf-pi0-star}
	\end{equation}
	Next apply the integral component of the operator, using \cite[Excercise 6.2.9]{Lefeuvre}. Indeed, applying Lemma~\ref{lem:integral-WF} together with $\supp(\chi') \subset [\eps/2, \eps]$ and \eqref{eq:wf-pi0-star} yields
	\begin{align*}
		&\WF\l(\int_{0}^{\infty} \chi'(t) \varphi^*_{-t} dt \pi_m^* h\r)\\
		&\subset \{(x,\xi) \in T^*SM : \xi(F(x)) = 0 \text{ and } \exists t \in \supp(\chi') \text{ s.t. } \Phi_{-t}(x,\xi) \in \WF(\pi_{m}^{*}u)\}\\
		&\subset \{\Phi_{t}(x,\xi): t \in [\eps/2, \eps], (x, \xi) \in \H^*\}.
	\end{align*}
    where $\Phi_t$ is the Hamiltonian lift of the flow and in the final step we used that $(x,\xi) \in \FF^* \oplus \H^*$ and $\xi(F(x)) = 0$ implies $(x,\xi) \in \H^*$.
	Next we will apply the holomorphic part of the resolvent $R_{+}^{\hol}(0)$ with the case of $R_{-}^{\hol}(0)$ being similar. Importantly, it is a result of Dyatlov and Zworski \cite[Prop 3.3]{2016-dyatlov-zworski} that allows us to control the wavefront set of $R_{+}^{\hol}(0)v$ by the wavefront of the distribution $v$. In particular, Corollary~\ref{lem:WF-of-Ru} promises
	\begin{align*}
		      &\WF\l(R^{\hol}_+(0)v\r)
			\subset \{(x,-\xi): (x, \xi) \in \WF(v)\}\\ %from \Delta T^*M term
			&\cup \{(x,\xi) \in T^*SM \sm \{0\} : \exists (y, -\eta) \in \WF(v) \text{ s.t. } \eta(F(x)) = 0, (x,\xi) = \Phi_t(y,\eta)\}\\ %from \Omega_+ term\\
			&\cup \{(x,\xi) \in T^*SM \sm \{0\} : \exists (y,-\eta) \in \WF(v) \text{ s.t. } (x,\xi,y,-\eta) \in E_u^* \times E_s^*\}.
	\end{align*}
    We would like to apply this for $v = \int_{0}^{\infty} \chi'(t) \varphi^*_{-t} dt \pi_m^* u$ with $\WF(v) \subset \{\Phi_t(x,\xi): t \in [\eps/2, \eps], (x, \xi) \in \H^*\}$. Consequently, all $(x, \xi) \in \WF(v)$ can be written as $\Phi_t(x,\eta)$ for $\eta(F(x)) = 0$ and therefore the first set in the above union can be absorbed into the second set in the above union. Furthermore, by the transversality assumption we have $(E_s \oplus \FF) \cap \V = \{0\}$, implying $E_s^* \cap \H^* = \{0\}$. But then by invariance of $E_s^*$, we find that for $(x,\xi) \in \H^*$, we have $\Phi_t(x,\xi) \notin E_s^*$ for all $t \in \R$ and therefore the last set in the above union is empty. Therefore, evaluating the second set in the union with $\WF(v) \subset \{\Phi_{t}(x,\xi): t \in [\eps/2, \eps], (x, \xi) \in \H^*\}$ as found earlier, we find

    \begin{align*}
		&\WF\l(R^{\hol}_+(0)\int_{0}^{\infty} \chi'(t) \varphi^*_{-t} dt \pi_m^* u\r)\\
		&\subset \{(x,\xi) \in T^*SM \sm 0: \Phi_t(x,\xi) : t \geq \eps/2, (x,\xi) \in \H^*\}.
	\end{align*}
	Finally, we aim to show
	\begin{equation*}
		\WF\l(\pi_{m*}R^{\hol}_+(0)\int_{0}^{\infty} \chi'(t) \varphi^*_{-t} dt \pi_m^* u\r) = \emptyset.
	\end{equation*}

	By \eqref{eq:pushforward-wf}, it suffices to have $\Phi_t(x,\xi) \notin \FF^* \oplus \H^*$ for $t \geq \eps/2$ if $(x,\xi) \in \H^*$, which is exactly the condition that the flow $\varphi_t$ has no conjugate points. 
    \end{proof}

\begin{lemma} 
    For a thermostat flow $\varphi_t$, the operator
    \begin{equation*}
        N \defeq \int_{S_xM} \int_{-\eps}^{\eps} \varphi_t^* \pi_0^* dt dS_x
    \end{equation*}
    is an elliptic pseudodifferential operator of order $-1$. 
    \label{thm:thermostat-psi-do}
\end{lemma}

\begin{proof}
    This follows from a change of coordinates together with \cite[Lemma B.1]{DPSU07} as done in \cite{FSU08}. Indeed, denote $\gamma_{x,v}(t) \defeq \pi_0 \circ \varphi_t(x,v)$ the curve generated by the flow with initial position $x$ and initial velocity $v$ and consider the change of variables $(t,v) \mapsto \gamma_{x,v}(t)$. Next, use the flow to obtain exponential chart $T_xM \to M$, which we denote $\xi \mapsto \exp_x(\xi)$, i.e., $\exp_{x}(\xi)$ is just the projection onto the manifold of the flow. In these exponential coordinates, using the metric $g$ to define the norm on $T_xM$, we can convert to spherical coordinates $(r,\omega)$. This process gives the change of variables $(t,v) \mapsto (r,\omega)$ and we denote the Jacobian by
    \begin{equation*}
        J = \det\frac{\partial(r,\omega)}{\partial(t,v)}
    \end{equation*}
    After identifying $f$ with its pullback under this coordinate change then extending $f$ by $0$, our operator becomes
    \begin{equation}
        Nf = \int_{S^n} \int_0^{R(\omega)} J^{-1}(x,r,\omega) f(x+r\omega) dr d\omega.
        \label{eq:normal-op-expression}
    \end{equation}
    Therefore, by \cite[Lemma B.1]{DPSU07} the normal operator \eqref{eq:normal-op-expression} is a classical pseudo-differential operator of order $-1$ with principal symbol
    \begin{equation}
        a_0(x,\xi) = 2\pi \int_{S_x U} J^{-1}(x,0,\omega) \delta(\omega \cdot \xi) d\sigma_x(\omega),
    \end{equation}
    where $U$ is an open set. Finally, a direct computation gives $J(x,0,\omega) = 1$, confirming that $N$ is an elliptic pseudodifferential operator.

\end{proof}

    \begin{proof}[Proof of Theorem~\ref{thm:main-thermostat-thm}]
    Let $\varphi_t \colon SM \to SM$ denote the thermostat flow over the unit tangent bundle of a manifold $M$. We then must compute the normal operator $\Pi_0 = -\pi_{0*}(R^{\hol}_+(0) + R^{\hol}_-(0)+ 1 \otimes 1)\pi_0^*$. By Lemma~\ref{lem:resolvent-comp} we have
    \begin{equation}
        R^{\hol}_{\pm}(0) = \int_0^{\infty} \chi(t) \varphi^*_{\mp t}dt + R_{\pm}^{\hol}(0)\int_0^{\infty} \chi'(t) \varphi^*_{\mp t}dt + \text{smoothing}.
        \label{eq:resolvent-decomp}
    \end{equation}
    Indeed, the last term given in \eqref{eq:hol-res-expression} is smoothing because the spectral projections $\Pi_0^{\pm}$ simply integrate over the SRB measures. In addition, by Lemma~\ref{lem:smoothing-term}, the operator
    \begin{equation*}
        \pi_{m*}R_{\pm}^{\hol}(0)\int_0^{\infty} \chi'(t) \varphi^*_{\mp t}dt\pi_m^*
    \end{equation*}
    is smoothing, which is precisely the second term of \eqref{eq:resolvent-decomp} after pre-composing and post-composing by $\pi_m^*$ and $\pi_{m*}$ respectively. Therefore we may write the normal operator $\Pi_0$ as follows:
    \begin{align*}
        \Pi_0 &= \pi_0^* (R_+^{\hol}(0) + R_-^{\hol}(0)) \pi_0^*  +  \text{smoothing}\\
        &= \pi_0^* \left( \int_0^{\infty} \chi(t) \varphi^*_{-t}dt + \int_0^{\infty} \chi(t) \varphi^*_{t}dt \right) \pi_{0*} + \text{smoothing}\\
        &= \pi_{0*} \int_{-\infty}^{\infty} \chi(t) \varphi_{t}^* \pi_0^*  dt + \text{smoothing}.
    \end{align*}
    We hope to show $\Pi_0$ is a pseudodifferential operator of order $-1$, so have to study the singularities of the Schwartz kernel $K$ of $\Pi_0$. Indeed, the Schwartz kernel $K(x, \cdot)$ for some fixed $x \in M$ is reflected in the operator operator $f \mapsto \Pi_0 f (x)$ for some $x \in M$ which, up to smoothing, is simply
    \begin{equation*}
    \int_{S_xM}\int_{-\eps}^{\eps} \varphi_t^* \pi_0^* dt dS_x.
    \label{eq:anosov-normal-op}
    \end{equation*}
    That is, by taking $\eps$ small enough, the analysis of the singularities of the Schwartz kernel of $\Pi_0$ at each $x \in M$ is equivalent to the analysis of the corresponding normal operator
    \begin{equation}
    \int_{S_xM'}\int \varphi_t^* \pi_0^* dt dS_x: C^{\infty}(SM') \to C^{\infty}(SM')
    \label{eq:normal}
    \end{equation}
    over $M' := \{\gamma_{x,v}(t) : t \in [-\eps, \eps], v \in S_xM\}$. By Lemma~\ref{thm:thermostat-psi-do}, this operator is an elliptic pseudodifferential operator of order $-1$.
    \end{proof}

\section{Magnetic Potential-Solenoidal Decomposition}
\label{sec:mag_ps}

In this section we work with magnetic flows, introduced in Section \ref{sec:thermostats}. In order to help us to understand the magnetic ray transform and the magnetic normal operator, we are interested in generalizing the potential-solenoidal decomposition available in the Riemannian setting. To this end, first following \cite{OPS25}, one can extend the Lorentz force to tensors as follows:
\begin{equation} \label{eq:Lorentz_tensors}
    Y(T)_x (v_1, \ldots, v_m )=\frac{1}{m} (T_x (Y (v_1 ), v_2, \ldots, v_m )+\cdots+T_x (v_1, \ldots, Y (v_m ))),
\end{equation}
where $T \in C^{\infty} (S^{m} (T^* M))$, $v_1, \ldots, v_m \in T_x M$, and $Yf=f$ for any $f \in C^{\infty}(M)$. Following \cites{DPSU07, OPS25}, we define the \emph{magnetic potential} by
\begin{align*}
    D_{\mu} \colon C^{\infty}(S^m (T^* M )) \times C^{\infty}(S^{m-1}(T^* M )) &\to C^{\infty}(S^{m+1} (T^* M )) \times C^{\infty}(S^{m}(T^* M )), \\
    [\xi,\eta] & \mapsto [D \xi+(m-1) Y(\eta) \otimes g,D \eta+m Y(\xi)],
\end{align*}
where $D$ is the symmetrized covariant derivative given by 
    \[ D:=\mathcal{S} \circ \nabla \colon C^{\infty} (M, S^{m} (T^{*}M)) \to C^{\infty}(M, S^{m+1} (T^{*} M)). \]
Here, as usual, $\nabla$ denotes the Levi-Civita connection, and $\mathcal{S}$ the symmetrization. In this way, $D$ maps symmetric tensors of order $m$ to symmetric tensors of order $m+1$. We also write $D_{\mu}$ as the following matrix valued differential operator:
    \begin{equation} \label{eq:matrix_mag_div}
        D_{\mu}=\begin{pmatrix}
        D & (m-1)Y(\bullet ) \otimes g \\ mY & D
    \end{pmatrix}.
    \end{equation}

For $[p,q],[r,s] \in C^{\infty}(S^m (T^* M )) \times C^{\infty}(S^{m-1}(T^* M ))$, we define its $\mathbf{L}^{2}$ product by
\[  \langle [p,q],[r,s] \rangle_{\mathbf{L}^{2}}=\langle p,r \rangle_{L^{2}(M,T^{*}M^{\otimes m},\vol_{g})}+\langle q,s \rangle_{L^{2}(M,T^{*}M^{\otimes (m-1)},\vol_{g})}. \]
We recall that the space $L^{2}(M,T^{*}M^{\otimes m},\vol_{g})$ is endowed with the usual $L^{2}$ product on tensors induced by the Riemannian metric, see \cite{Lefeuvre}*{Section 14.1}. We denote by $\mathbf{L}^{2}$ the space of pairs of symmetric tensors with finite $\mathbf{L}^{2}$-norm. In a similar fashion, we will write $\mathbf{H}^{s}$ to denote the space of pairs of symmetric tensors with finite $\mathbf{H}^{s}$-norm, where the norm is defined in the obvious way. We omit the order of the pair in order to simplify the notation, but we mention it when it is necessary to avoid confusions. 

Our first task will be the study of the operator $D_{\mu}$. As we will show later, this will be useful to understand the kernel of $I_{m}$, and the structure of pairs of symmetric tensors.

\begin{lemma} \label{lemma:adjoint}
    The formal adjoint of $D_{\mu}$ with respect to the $\mathbf{L}^2$ product is 
    \[ D_{\mu}^{*}\colon  C^{\infty}(S^{m} (T^* M )) \times C^{\infty}(S^{m-1}(T^* M )) \to C^{\infty}(S^{m-1} (T^* M )) \times C^{\infty}(S^{m-2}(T^* M )), \]
    given by
    \[ D_{\mu}^{*}=\begin{pmatrix}
        -\tr(\nabla \bullet) & -(m-1)Y \\
         -m\tr Y & -\tr(\nabla \bullet)
    \end{pmatrix}. \]
    We call $D_{\mu}^{*}$ the \emph{magnetic divergence}.
\end{lemma}

\begin{proof}
Let $[p,q],[r,s] \in C^{\infty}(S^m (T^* M )) \times C^{\infty}(S^{m-1}(T^* M ))$, where $[r,s]=D_{\mu}[\xi,\eta]^{\top}$. Then,
\begin{align*}
    \langle [p,q],[r,s]\rangle _{\mathbf{L}^{2}} &=\int_{M} \{ (p,D\xi)+(m-2)(p,Y(\eta) \otimes g) +(q,D\eta)+(m-1)(q,Y(\xi) ) \} \\
    &=\int_{M} \{(D^{*}p,\xi)+(m-2)(p,Y(\eta) \otimes g) +(D^{*}q,\eta) +(m-1) (q,Y(\xi))\},
\end{align*}
where we used the adjoint of $D$ is $D^{*}:=-\tr(\nabla \bullet)$, see for instance \cite{Lefeuvre}*{Lemma 14.1.8}. To deal with the other terms, let $\{e^{j}\}_{j=1}^{n}$ be an orthonormal frame on $TM$. Then,  using \eqref{eq:Lorentz_tensors}
\begin{align*}
    Y(\xi)_{j_{1} \ldots j_{m-1}}&=Y(\xi)(e_{j_{1}},\ldots,e_{j_{m-1}}) \\
    &=\frac{1}{m-1}(\xi_x (Y (e_{j_{1}} ), e_{j_{2}}, \ldots, e_{j_{m-1}} )+\cdots+\xi_x (e_{j_{1}}, \ldots, Y (e_{j_{m-1}} ))) \\
    &=\frac{1}{m-1}\sum_{\ell=1}^{m-1}Y_{j_{\ell}}^{k}\xi_{j_{1}\ldots k \ldots j_{m-1}}.
\end{align*}
On the other hand, since $Y$ is antisymmetric, we have
\begin{align}
    \int_{M}& g^{i_{1}j_{1}} \cdots g^{i_{m-1}j_{m-1}} q_{i_{1}\ldots i_{m-1}}  Y^{k}_{j_{1}}\xi_{kj_{2}\ldots j_{m-1}} \\
    &=-\int_{M} g^{i_{1}j_{1}} \cdots g^{i_{m-1}j_{m-1}} Y_{i_{1}}^{k}q_{k\ldots i_{m-1}}  \xi_{j_{1}j_{2}\ldots j_{m-1}} .
\end{align}
Thus, 
\begin{align*}
    \int_{M} ( q,Y(\xi) ) =&\int_{M} g^{i_{1}j_{1}} \cdots g^{i_{m-1}j_{m-1}} q_{i_{1}\ldots i_{m-1}}Y(\xi)_{j_{1} \ldots j_{m-1}}  \\
    =&\frac{1}{m-1}\int_{M} g^{i_{1}j_{1}} \cdots g^{i_{m-1}j_{m-1}} q_{i_{1}\ldots i_{m-1}} \times \\
    &\sum_{\ell=1}^{m-1}Y_{j_{\ell}}^{k}\xi_{j_{1}\ldots k \ldots j_{m-1}} \\
    =&\int_{M} g^{i_{1}j_{1}} \cdots g^{i_{m-1}j_{m-1}} (-Y(q))_{i_{1} \ldots i_{m-1}} \xi_{j_{1}\ldots j_{m-1}}  \\
    =&\int_{M} (-Y(q),\xi ).
\end{align*}
From this follows that the adjoint of $Y$ (acting on $(m-1)$-tensors) is $-Y$. Finally, note that if $m=2$, we have that $(m-2)Y(\eta)\otimes g=0$, while if $m \neq 2$ using $Y^{*}=-Y$ and the fact that $(m-2)Y(\eta)\otimes g=mY(\eta \otimes g)$ (which follows from \cite{OPS25}*{Equation (3.18)} with $k=1$ and $n=0$), we see that
\begin{align*}
    \int_{M} (m-2)(p,Y(\eta) \otimes g)  &=\int_{M} ( mp,Y(n \otimes g) ) \\
    &=\int_{M} ( -mY(p),\eta \otimes g ) \\
    &=\int_{M} (-m\tr Y(p),\eta ). 
\end{align*}
Therefore, 
\[ \langle [p,q],[r,s] \rangle_{\mathbf{L}^{2}}=\int_{M} \{ ( D^{*}p+(1-m)Y(q),\xi )+( D^{*}q-m\tr Y(p),\eta) \}d\vol, \]
which proves the claim.    
\end{proof}

\begin{rmk} \label{rmk:div_m=2}
    In the particular case where $D_{\mu}^{*}$ acts on pairs $[p,q]$ where $p$ is a 2-tensor and $q$ a 1-form, from the proof we conclude that 
    \[ D_{\mu}^{*}=\begin{pmatrix}
        -\tr(\nabla \bullet) & -Y \\ 0 & -\tr(\nabla \bullet)
    \end{pmatrix},\]
    which coincides up to a factor on the top-right entry with the adjoint operator from \cite{DPSU07}*{Equation (3.23)}. The factor that appears there is added since was useful to obtain some bounds in that previous work.
\end{rmk}

%QUESTION: why the ^{top} here?
\begin{defin}
    The pair $[p, q] \in C^{\infty}(S^m (T^* M )) \times C^{\infty}(S^{m-1}(T^* M ))$ is called \emph{potential} if $[p,q]=D_{\mu}[\xi,\eta]^{\top}$ for some $[\xi,\eta] \in C^{\infty}(S^{m-1} (T^* M )) \times C^{\infty}(S^{m-2}(T^* M ))$.
\end{defin}

Next, we obtain an identity that helps us to relate $F$ with $D_{\mu}$. In this way, we are able relate potential tensors with the kernel of the ray transform.

\begin{lemma} \label{lemma:commutating}
    $F [\pi_{m}^{*},\pi_{m-1}^{*}]=[\pi_{m+1}^{*},\pi_{m}^{*}]D_{\mu}$. In particular, $\mathrm{ran}D_{\mu} \subset \ker I_{m}$.
\end{lemma}

We recall that from the introduction $I_{m}[p,q]=I(\pi_{m}^{*}p+\pi_{m-1}^{*}q)$. The proof is essentially the same as in \cite{OPS25}*{Lemma 5.1}, so we omit it. The case of equality on Lemma \ref{lemma:commutating} is of particular interest.

\begin{defin}
    We say that $I_{m}$ is \emph{solenoidal injective} (abbreviated as \emph{s-injective}) if $\mathrm{ran}D_{\mu} = \ker I_{m}$.
\end{defin}

Being s-injective means that the whole kernel of the ray transform are the ones given by Lemma \ref{lemma:commutating}, i.e., pairs of tensors that can be written as potentials. The name \emph{solenoidal} will have more sense (and is explained) after Theorem \ref{thm:ps_decomposition}.

Now we show that $D_{\mu}$ is an elliptic matrix valued operator, and characterize its kernel. We recall a map that will be useful for this purpose:
\begin{align*}
    \mathcal{J} \colon S^{j}(T^{*}M) &\to S^{j+2}(T^{*}M), \\
    u & \mapsto \mathcal{S}(g\otimes u).
\end{align*}

\begin{lemma} \label{lemma:ellipticity}
    The operator 
    \[ 
    \begin{split}
        D_{\mu} \colon &C^{\infty}(M,S^{m}(T^{*}M)) \times C^{\infty}(M,S^{m-1}(T^{*}M)) \\
        &\to C^{\infty}(M,S^{m+1}(T^{*}M)) \times C^{\infty}(M,S^{m}(T^{*}M)),
    \end{split}
    \]
    is elliptic. Furthermore, for $[p,q] \in \ker D_{\mu}$, we have
    \begin{itemize}
        \item if $m$ is odd, then $p=0$ and $q=J^{(m-1)/2}a$ where $a$ is constant;
        \item if $m$ is even, then $q=0$ and $p=J^{m/2}a$, where $a$ is constant.
    \end{itemize}
\end{lemma}

\begin{proof}
The ellipticity is straight forward. Indeed, the principal symbol of $D_{\mu}$ is 
\[ \sigma_{D_{\mu}}(x,\xi)=\begin{pmatrix}
    \sigma_{D} & 0 \\ 0 & \sigma_{D}
\end{pmatrix}=\begin{pmatrix}
    ij_{\xi}^{m} & 0 \\ 0 & ij_{\xi}^{m-1}
\end{pmatrix}, \]
where $j_{\xi}^{k}=\mathcal{S}(\xi \otimes \bullet)$ is the symmetric multiplication by $\xi$ over tensors of order $k$. Since $D$ is elliptic, we conclude that $D_{\mu}$ is elliptic as well.

Now, let us study the kernel of $D_{\mu}$. By elliptic regularity, if $[p,q] \in \ker D_{\mu}$, then $[p,q]$ is smooth. Using Lemma \ref{lemma:commutating}, we see that $F(\pi_{m}^{*}p+\pi_{m-1}^{*}q)=0$. Since the magnetic flow is transitive, there exists a constant $c$ so that
\begin{equation} \label{eq:pi=c}
    \pi_{m}^{*}p(x,v)+\pi_{m-1}^{*}q(x,v)=c.
\end{equation}
First, let us assume that $m$ is odd. Then, evaluating \eqref{eq:pi=c} at $-v$ gives
\begin{equation} \label{eq:pi=c_minus}
    -\pi_{m}^{*}p(x,v)+\pi_{m-1}^{*}q(x,v)=c.
\end{equation}
It then follows from \eqref{eq:pi=c} and \eqref{eq:pi=c_minus} that $\pi_{m}^{*}p=0$, and hence $p=0$. Then, $\pi_{m-1}^{*}q=c$. It is known \cite{Lefeuvre}*{Section 14.1.1.3} that $\pi_{m-1}^{*}$ is an isomorphism between $\mathcal{J}^{(m-1)/2}(S^{0}T^{*}M)$ and the constants. Therefore, we conclude that $q=\mathcal{J}^{(m-1)/2}(a)$, for some $a \in C^{\infty}(M)$. Since we are working on $SM$, from $F\pi_{m-1}^{*}q=0$ we conclude $F\pi_{0}^{*}a=0$. This implies that $a$ is a constant function (for example, this follows from ergodicity, since the magnetic flow preserves the Liouville measure). The case with $m$ even is similar and we omit it.
\end{proof}

Finally, we have all the ingredients in order to state and prove the main result of this section.

\begin{thm} \label{thm:ps_decomposition}
Let $s \in \R$ and $f \in H^{s}(M, S^{m} (T^{*} M)) \times H^{s}(M, S^{m-1} (T^{*}M))$. Then, there exists a unique pair of pairs of tensors
\begin{align*}
    P &\in H^{s+1}(M,S^{m-1}(T^{*}M)) \times H^{s+1}(M,S^{m-2}(T^{*}M)), \\
    H &\in H^s (M, S^{m+1}(T^{*}M)) \times H^s (M, S^{m}(T^{*}M)),
\end{align*}
        such that $f=D_{\mu}P+H$, and $D_{\mu}^{*}H=0$.
\end{thm}

\begin{proof}
Define 
\[ K_{m}=\begin{cases} \mathcal{J}^{(m-1)/2}(1) & \text{if $m$ is odd}, \\ \mathcal{J}^{m/2}(1) & \text{if $m$ is even},\end{cases} \]
and 
\[ \Pi_{K_{m}}=\begin{cases}
    \langle \bullet, [0,K_{m}]\rangle [0,K_{m}] & \text{if $m$ is odd,} \\
    \langle \bullet, [K_{m},0]\rangle [K_{m},0] & \text{if $m$ is even,}
\end{cases} \]
be the orthogonal projection onto $\ker D_{\mu}$. Finally, set $\Delta_{m}=D_{\mu}^{*}D_{\mu}+\Pi_{K_{m}}$. By the properties of $D_{\mu}$ and $D_{\mu}^{*}$ (Lemma \ref{lemma:ellipticity}), $\Delta_{m}$ is is an elliptic differential operator of order $2$. Therefore, its inverse $\Delta_{m}^{-1}$ is a pseudodifferential operator of order $-2$. This allow us to define 
\begin{equation} \label{eq:pi_sol}
    \pi_{\ker D_{\mu}^{*}}:=Id-D_{\mu}\Delta_{m}^{-1}D_{\mu}^{*}.
\end{equation}
We set $H=\pi_{\ker D_{\mu}^{*}}f$ and $P=\Delta_{m}^{-1}D_{\mu}^{*}f$. It is immediate that $H$ and $P$ defined in this way satisfy the required properties.
\end{proof}

In the literature, the decomposition from Theorem \ref{thm:ps_decomposition} is called \emph{potential-solenoidal decomposition}. Elements in $\mathrm{ran}D_{\mu}$ are called \emph{potential tensors}, while the elements in $\ker D_{\mu}^{*}$ are called \emph{solenoidal tensors}.

We conclude this section with an observation that will be useful later. Since $\sigma_{D_{\mu}}$ is injective for $\xi \neq 0$, we have the following decomposition
\[
\begin{split}
    S^{m}(T_{x}^{*}M) \times S^{m-1}(T_{x}^{*}M) =&\mathrm{ran}(i\sigma_{D_{\mu}}|_{S^{m-1}(T_{x}^{*}M) \times S^{m-2}(T_{x}^{*}M})) \\
    &\oplus \ker (i \sigma_{D_{\mu}^{*}}|_{S^{m}(T_{x}^{*}M) \times S^{m-1}(T_{x}^{*}M)}) \\
    =&\mathrm{ran} \left( \begin{pmatrix}
        j_{\xi}^{m-1} & 0 \\ 0 & j_{\xi}^{m-2}
    \end{pmatrix}\bigg|_{S^{m-1}(T_{x}^{*}M) \times S^{m-2}(T_{x}^{*}M)} \right)  \\
    &\oplus \ker \left( \begin{pmatrix}
        \iota_{\xi^{\sharp}}^{m} & 0 \\ 0 & \iota_{\xi^{\sharp}}^{m-1}
    \end{pmatrix}\bigg|_{S^{m}(T_{x}^{*}M) \times S^{m-1}(T_{x}^{*}M)}   \right),
\end{split}
\]
where as in the proof of Lemma \ref{lemma:ellipticity}, the super-index is used to denote the order of the tensors where the operators act.

Let $\pi_{\ker_{i_{\xi}}}$ denote the projection onto the right space on the right-hand side. Observe that $\sigma_{\pi_{\ker D_{\mu}^{*}}}=\pi_{\ker_{i_{\xi}}}$.

\section{The Magnetic Normal Operator} \label{sec:magnetic_normal}

In this section we generalize the normal operator to magnetic flows. We study its analytic structure and obtain some properties that follow by its psuedodifferential behavior.

\subsection{Ellipticity}

Recall that the magnetic normal operator was defined in the Introduction by \eqref{eq:magnetic_normal}. One should think of $N_{m}$ as the analog of the normal operator for the tensorial X-ray transform, but in the context of closed manifolds. As expected, it has a pseudoddiferential nature.

As in Section \ref{sec:mag_ps}, the super-indices over $i_{\xi}$ denote the order of tensors where it acts. Similarly as in the case of the geodesic flow, we have $C_{n,m}=\sqrt{\pi} \frac{\Gamma(\frac{n-1}{2}+m)}{\Gamma(\frac{n}{2}+m)}$.

\begin{proof}[Proof of Theorem \ref{thm:normal_magnetic_psido}]
We will follow \cites{SSU05, DPSU07, Lefeuvre}. Using Lemmas \ref{lem:smoothing1} and \ref{lemma:smoothing_middle}, is enough to show that 
\begin{equation} \label{eq:normal_tru}
    \pi_{m *} \circ \int_{-\eps}^{\eps}e^{tF}dt \circ \pi_{m}^{*}
\end{equation}
is a classical pseudodifferential operator of order $-1$. We take $\eps>0$ smaller than the magnetic injectivity radius, so that the magnetic exponential map in a ball of radius $\eps$ is a diffeomorphism. Let us fix a local chart $(U,\psi)$. Let $\chi$ be a cutoff with support in $\psi(U)$ such that $e^{t F}(\operatorname{supp}(\chi)) \subset \psi(U)$ for all $t \in(-\eps, \eps)$, then for $f \in C_{\mathrm{comp}}^{\infty}(\varphi(U))$:
\begin{align*}
    &\left(\left(\chi \pi_{m *} \int_{-\eps}^{\eps} e^{t X} dt \pi_m^* \chi\right) f \right)_{j_{1}\ldots j_{m}}(x) \\
    &=\int_{S_{x}M} \chi(x)v_{j_{1}}\cdots v_{j_{m}} \int_{-\eps}^{\eps}\pi_{m}^{*}(\chi f)(\varphi_{t}(x,v))dtdv \\
    &=\int_{S_{x}M} \chi(x)v_{j_{1}}\cdots v_{j_{m}} \int_{-\eps}^{\eps}\chi(\pi(\varphi_{t}(x,v)))f(\pi(\varphi_{t}(x,v)))_{i_{1}\ldots i_{m}} \dot{\gamma}^{i_{1}} \cdots \dot{\gamma}^{i_{m}} dtdv \\
    &=\int_{S_{x}M} \int_{-\eps}^{\eps} \chi(x)v_{j_{1}}\cdots v_{j_{m}} \chi(\gamma(t))f_{i_{1}\ldots i_{m}}(\gamma(t))\dot{\gamma}^{i_{1}} \cdots \dot{\gamma}^{i_{m}} dtdv,
\end{align*}
where we used \eqref{eq:pi_m^*}. Here, $\pi \colon TM \to M$ is the base point projection, $\gamma$ is such that $\varphi_{t}(x,v)=(\gamma(t),\dot{\gamma}(t))$, and $v_{i}=(v^{\flat})_{i}=g_{ij}v^{j}$. Since the magnetic exponential map is just a $C^{1}$ diffeomorphism, it is better to use the change of variables as in \cite{DPSU07}: set 
\[ \gamma(t)-x=tm(t,v;x), \quad m(0,v;x)=v, \]
and introduce the variables $(r,\omega) \in \R \times S_{x}M$ via
\[ r=t|m(t,v;x)|_{g}, \quad \omega=\frac{m(t,v;x)}{|m(t,v;x)|_{g}}.\]
We obtain
\begin{align*}
    &\left(\left(\chi \pi_{m *} \int_{-\eps}^{\eps} e^{t F} dt \pi_m^* \chi\right) f \right)_{i_{1}\ldots i_{m}}(x)  \\
    =&\int_{S_{x}M} \int_{-\eps}^{\eps} \chi(x)v_{i_{1}}(r,\omega;x)\cdots v_{i_{m}} (r,\omega;x) \chi(x+r\omega)f_{i_{1}\ldots i_{m}}(x+r\omega) \times \\
    &\times w(r,\omega;x)^{i_{1}} \cdots w(r,\omega;x)^{i_{m}} J^{-1}(r,\omega;x)drd\omega,
\end{align*}
where $w(r,\omega;x)=\dot{\gamma}(t)$ and $J$ is the determinant of the change of variables, i.e., $J=\det \partial (r,\omega)/\partial(t,v)$. Since our computations are local, \cite{DPSU07}*{Lemma B.1} gives that \eqref{eq:normal_tru} is a classical $\Psi$DO of order $-1$ and with principal symbol given by
\begin{equation} \label{eq:prin_symbol}
    2\pi\int_{S_{x}M}\omega_{j_{1}} \cdots \omega_{j_{m}} \omega^{i_{1}} \cdots \omega^{i_{m}} \delta(\omega \cdot \xi)d\omega=\frac{2\pi}{|\xi|}\pi_{m *}\pi_{m}^{*}\delta (\bullet \cdot \hat{\xi}),
\end{equation}
where we used the homogeneity of the Dirac delta distribution. Of course, the same works with $\pi_{m-1 *}\Pi \pi_{m-1}^{*}$. For $\pi_{m-1 *}\Pi \pi_{m}^{*}$ and $\pi_{m *}\Pi \pi_{m-1}^{*}$ the analysis is the same, but we obtain that their principal symbols are
\begin{align*}
    2\pi\int_{S_{x}M}\omega_{j_{1}} \cdots \omega_{j_{m-1}} \omega^{i_{1}} \cdots \omega^{i_{m}} \delta(\omega \cdot \xi)d\omega=0, \\
    2\pi\int_{S_{x}M}\omega_{j_{1}} \cdots \omega_{j_{m}} \omega^{i_{1}} \cdots \omega^{i_{m-1}} \delta(\omega \cdot \xi)d\omega=0,
\end{align*}
respectively, since the integrals are over odd functions. To finish the proof, we recall from the proof of Theorem 16.2.1 in \cite{Lefeuvre} that $\pi_{m}^{*}j_{\xi}=\langle \xi,\bullet \rangle\pi_{m-1}^{*}$ on $\{\langle \xi,v \rangle=0 \}$. When we integrate this against $\delta(\omega \cdot \xi)$, this term vanishes. Now, using the potential-solenoidal decomposition from Theorem \ref{thm:ps_decomposition}, we have for $k=1,2$
\begin{equation} \label{eq:tensor_decom}
    T_{k}=\begin{pmatrix}
    j_{\xi}^{m-1}a_{1}^{k} \\j_{\xi}^{m-2}a_{2}^{k}
\end{pmatrix}+B_{k}, \qquad B_{k}=\begin{pmatrix}
    b_{1}^{k} \\ b_{2}^{k}
\end{pmatrix} \in \ker \begin{pmatrix}
        \iota_{\xi^{\sharp}}^{m} & 0 \\ 0 & \iota_{\xi^{\sharp}}^{m-1}
    \end{pmatrix}.
\end{equation}
Therefore, the only part of the tensors that survives after the integration if the one in the kernel of $\mathrm{diag}(i_{\xi^{\sharp}},i_{\xi^{\sharp}})$. In a more explicit way, so far we have, for $T_{1}, T_{2}$ as in \eqref{eq:tensor_decom},
\[
\begin{split}
    \langle \sigma_{N_{m}}(x,\xi)T_{1},T_{2} \rangle = \frac{2\pi}{|\xi|} \bigg[ 
    & \int_{\langle \xi,v \rangle=0} \pi_{m}^{*}b_{1}^{1}(v)\pi_{m}^{*}b_{1}^{2} dS_{\xi}(v), \\
    & \int_{\langle \xi,v \rangle=0} \pi_{m-1}^{*}b_{2}^{1}(v)\pi_{m-1}^{*}b_{2}^{2} dS_{\xi}(v) \bigg].
\end{split}
\]
Hence, it is enough to show that 
\[ C_{n, m} \int_{\langle\xi, v\rangle=0} \pi_m^* f_1(v) \pi_m^* f_2(v) \mathrm{d} S_{\xi}(v)=\int_{S_{x}M \cap \{\langle \xi,v\rangle=0\}} \pi_m^* f_1(v) \pi_m^* f_2(v) \mathrm{d} S(v), \]
for any $f_{1},f_{2} \in \ker i_{\xi^{\sharp}}$, which is part of the proof of Theorem 16.2.1 \cite{Lefeuvre}. This shows that that the principal symbol of $\pi_{m *}\Pi \pi_{m}^{*}$ is given by
\[ \frac{2\pi}{C_{n,m}}\frac{1}{|\xi|}\pi_{\ker_{i_{\xi}}}\pi_{m *}\pi_{m}^{*}\pi_{\ker_{i_{\xi}}}, \]
finishing the proof.    
\end{proof}

Ellipticity implies that we can construct approximate inverses for this operator. Indeed, the usual construction of parametrices for elliptic $\Psi$DOs gives:

\begin{lemma} \label{lem:ellipticity}
    $N_{m}$ is elliptic on solenoidal tensors. More precisely, there exist $Q \in \Psi^{1}(M)$, $R \in \Psi^{-\infty}(M)$ such that 
    \[ QN_{m}=\pi_{\ker_{D_{\mu}^{*}}}+R. \]
\end{lemma}

\subsection{Injectivity and Surjectivity}

In this subsection we explore the injectivity and surjectivity of $N_{m}$, and how is related with the one of $I_{m}$. We begin with the next result.

\begin{lemma} \label{lemma:inj_equiv}
    $I_{m}$ is s-injective if and only if $N_{m}$ is injective on solenoidal tensors, that is, in $H_{\mathrm{sol}}^{s}=\ker D_{\mu}^{*}|_{H^{s}(S^{m}(T^{*}M)) \times H^{s}(S^{m}(T^{*}M))}$ for all $s \in \R$.
\end{lemma}

\begin{proof}
First assume that $N_{m}$ is injective. Let us consider $f=[p,q] \in C^{\infty}(S^{m}(T^{*}M)) \times C^{\infty}(S^{m-1}(T^{*}M))$ with $0=I_{m}f=I(\pi_{m}^{*}p+\pi_{m-1}^{*}q)$. Since we are assuming that the flow is Anosov, by Liv\v sic theorem (Lemma \ref{lemma:livsic}), there exists $u \in C^{\infty}(SM)$ with $\pi_{m}^{*}p+\pi_{m-1}^{*}q=Fu$. Now,
\[ N_{m}f=\begin{pmatrix}
    \pi_{m *}\Pi (\pi_{m}^{*}p+\pi_{m-1}^{*}q) \\ \pi_{m-1 *}\Pi (\pi_{m}^{*}p+\pi_{m-1}^{*}q) \end{pmatrix}=\begin{pmatrix}
    \pi_{m *}\Pi Fu \\ \pi_{m-1 *}\Pi Fu
    \end{pmatrix} =0, \]
where in the second equality we used that $Fu$ integrates to 0 in $SM$ (since $F$ preserves the Liouville form and $M$ has no boundary), and in the last step we used Lemma \ref{lemma:pi_properties} (i).  

To prove the converse, fix $s \in \R$ and let $f=[p,q] \in H_{\mathrm{sol}}^{s}$. By Lemma \ref{lem:ellipticity}, $f$ is smooth. The non-negativity of $\Pi$ (Lemma \ref{lemma:pi_properties}) allows us to write $\sqrt{\Pi}$. Then, 
\begin{align*}
    0=&\langle N_{m}f,f \rangle_{\mathbf{L}^{2}} \\
    =&\left \langle \begin{pmatrix}
        \Pi (\pi_{m}^{*}p+\pi_{m-1}^{*}q) \\ \Pi (\pi_{m}^{*}p+\pi_{m-1}^{*}q) 
    \end{pmatrix}, \begin{pmatrix}
        \pi_{m}^{*}p \\ \pi_{m-1}^{*}q
    \end{pmatrix} \right\rangle_{\mathbf{L}^{2}} \\
    &+\left \langle \begin{pmatrix}
        (1 \otimes 1) (\pi_{m}^{*}p+\pi_{m-1}^{*}q) \\ (1 \otimes 1)(\pi_{m}^{*}p+\pi_{m-1}^{*}q) 
    \end{pmatrix}, \begin{pmatrix}
        \pi_{m}^{*}p \\ \pi_{m-1}^{*}q
    \end{pmatrix} \right\rangle_{\mathbf{L}^{2}} \\
    =& \int_{SM}(\sqrt{\Pi}\pi_{m}^{*}p+\sqrt{\Pi}\pi_{m-1}^{*}q)^{2}+\int_{SM}(\pi_{m}^{*}p+\pi_{m-1}q)^{2}
\end{align*}
It follows that both terms on the right-hand side of the last equality are zero. Thus, $\Pi (\pi_{m}^{*}p+\pi_{m-1}^{*}q)=0$. An application of Lemma \ref{lemma:pi_properties} (iii) implies that there exists $u \in C^{\infty}(SM)$ and $v \in \ker F$ so that $v+Fu=\pi_{m}^{*}p+\pi_{m-1}^{*}q$. However, $v$ is constant since $F$ preserves the Liouville form. Hence, using Stokes' theorem and that $F$ preserves the Liouville form, we obtain
\[ 0=\int_{SM}(\pi_{m}^{*}p+\pi_{m-1}q)=\int_{SM}v \]
which shows that $v=0$. We conclude that $Fu=\pi_{m}^{*}p+\pi_{m-1}^{*}q$, which shows that $f=[p,q] \in \ker I_{m}$. Since $f$ is solenoidal, s-injectivity of $I_{m}$ implies that $f \equiv 0$.    
\end{proof}

Since s-injectivity of $I_{m}$ is equivalent to that of $N_{m}$, and this last operator is elliptic on solenoidal tensors (Lemma \ref{lemma:inj_equiv}), and hence Fredholm, we conclude:

\begin{lemma}
The kernel of $I_m$ on the space of smooth solenoidal tensors is finite dimensional. 
\end{lemma}

We also have the following

\begin{lemma}
If $I_m$ is s-injective, then there exists $Q' \in \Psi^{1}(M)$such that: $Q' N_m=\pi_{\ker D_{\mu}^{*}}$.
\end{lemma}

The proof is a verbatim of the proof of \cite{Lefeuvre}*{Theorem 16.2.7}, but we write it for completeness.

\begin{proof}
    As before, 
\[
\begin{split}
    N_{m} \colon &(H^{s}(S^{m}(T^{*}M)) \times H^{s}(S^{m}(T^{*}M)))_{\mathrm{sol}} \\
    &\to (H^{s+1}(S^{m}(T^{*}M)) \times H^{s+1}(S^{m}(T^{*}M)))_{\mathrm{sol}},
\end{split}
\]
is a Fredholm operator for all $s \in \R$. It is (formally) self-adjoint, which implies that its index is zero. By the hypothesis on $I_{m}$, Lemma \ref{lemma:inj_equiv} gives injectivity of $N_{m}$. Therefore, $N_{m}$ is invertible. Let $Q$ and $R$ be as in Lemma \ref{lem:ellipticity}, and write $Q' \defeq \pi_{\ker D_{\mu}^{*}}N_{m}^{-1}\pi_{\ker D_{\mu}^{*}}$. Then, for $N_{m}$ between these spaces, we have
\[
\begin{split}
    Q'+RQ' =(\pi_{\ker D_{\mu}^{*}}+R)Q'=QN_{m}Q' =Q\pi_{\ker D_{\mu}^{*}},
\end{split}
\]
where in the first step we used that $\pi_{\ker D_{\mu}^{*}}^{2}=\pi_{\ker D_{\mu}^{*}}$, and last step we used that $N_{m}\pi_{\ker D_{\mu}^{*}}=N_{m}$. Then, $Q=Q\pi_{\ker D_{\mu}^{*}}+ \Psi^{-\infty}$. Since $Q \in \Psi^{1}$ and $\pi_{\ker D_{\mu}^{*}} \in \Psi^{0}$ (which follows from \eqref{eq:pi_sol}), we obtain that $Q \in \Psi^{1}$. Finally, since $N_{m}$ maps solenoidal tensors into solenoidal tensors, $\pi_{\ker D_{\mu}^{*}}N_{m}=N_{m}$, and it follows that $Q'N_{m}=\pi_{\ker D_{\mu}^{*}}$.
\end{proof}

As a consequence, using the mapping properties of pseudodifferential operators, we have the following stability estimate that will be useful for the proof of Theorem \ref{thm:stability_xray} in Section \ref{sec:inj_stab_xray}.

\begin{prop} \label{prop:cont}
If $I_m$ is s-injective, then for any $s \in \R$, there exists a constant $C:=C(s)>0$ such that: for all solenoidal tensors $f \in H^s (M, S^{m} (T^{*}M) ) \times H^s (M, S^{m-1} (T^{*}M) )$
\[ \|f\|_{\mathbf{H}^s} \leq C \|N_m f \|_{\mathbf{H}^{s+1}}. \]
\end{prop}

To understand when $N_{m}$ is surjective, we make the following observation. Since $N_{m}$ is formally a self-adjoint, elliptic on solenoidal tensors, it is a Fredholm operator of order 0, thus it is injective if and only if it is surjective (on solenoidal tensors). Furthermore, we have:

\begin{lemma}
    $I_{m}$ is s-injective if and only if
    \begin{align*}
        \tau_{m} \colon \mathcal{D}'_{\mathrm{inv}}(SM) &\to (C^{\infty}(M,S^{m}(T^{*}M)) \times C^{\infty}(M,S^{m-1}(T^{*}M)))_{\mathrm{sol}}, \\
        h &\mapsto [\pi_{m *}h,\pi_{m-1 *}h]^{\top},
    \end{align*}
    is surjective.
\end{lemma}

Here, $\mathcal{D}'_{\mathrm{inv}}(SM)$ denotes the set of distributions that are invariant under the magnetic flow, that is,
\[ \mathcal{D}'_{\mathrm{inv}}(SM)=\{ u \in \mathcal{D}'(SM):Fu=0\}. \]

\begin{proof}
    In first place, observe that the map is well defined, since
\begin{align*}
    \left\langle D_{\mu}^{*}\tau_{m}(h), \begin{bmatrix}
    r \\ s \end{bmatrix}
 \right\rangle_{\mathbf{L^{2}}}&=\left\langle \tau_{m}(h), D_{\mu}\begin{bmatrix}
    r \\ s \end{bmatrix}
 \right\rangle_{\mathbf{L^{2}}} \\
 &=\left \langle h, [\pi_{m}^{*},\pi_{m-1}^{*}]D_{\mu}\begin{bmatrix}
    r \\ s \end{bmatrix}
 \right\rangle_{L^{2}} \\
 &=\left\langle h,F[\pi_{m-1}^{*},\pi_{m-2}^{*}]\begin{bmatrix}
    r \\ s \end{bmatrix} \right\rangle_{L^{2}} \\
    &=-\left\langle Fh,[\pi_{m-1}^{*},\pi_{m-2}^{*}]\begin{bmatrix}
    r \\ s \end{bmatrix} \right\rangle_{L^{2}}=0,
\end{align*}
where the third equality follows from Lemma \ref{lemma:commutating}.

Now we show the statement. Let us assume first that $\tau_{m}$ is surjective. So, given $f=[p,q] \in \ker I_{m}$, there exists an invariant distribution $h$ such that $\tau_{m}h=f$. On the other hand, since $f \in \ker I_{m}$, by Liv\v sic theorem (Lemma \ref{lemma:livsic}), $[\pi_{m}^{*},\pi_{m-1}^{*}]f=Fu$. Then, 
\[ 0=\langle Fh,u\rangle_{L^{2}}=-\langle h,Fu \rangle_{L^{2}}=-\langle h, [\pi_{m}^{*},\pi_{m-1}^{*}]f\rangle_{L^{2}}=-\langle \tau_{m}h,f\rangle_{\mathbf{L}^{2}} =-\|f\|_{\mathbf{L}^{2}}^{2}. \]
Conversely, if $I_{m}$ is s-injective, then by Lemma \ref{lemma:inj_equiv} $N_{m}$ is injective on solenoidal tensors. As we mentioned before, this is equivalent to the surjectivity of $N_{m}$ on solenoidal tensors. So, given $f_{1}$ solenoidal, there exists a solenoidal $f_{2}$ so that $f_{1}=N_{m}f_{2}=\tau_{m}((\Pi+1\otimes 1) [\pi_{m}^{*}, \pi_{m-1}^{*}]f_{2})$. The fact that $(\Pi+1\otimes 1) [\pi_{m}^{*}, \pi_{m-1}^{*}]f_{2}$ is invariant follows from Lemma \ref{lemma:pi_properties} (i).
\end{proof}

\begin{prop} \label{prop:bound_tau}
    The operator
    \[ \Pi [\pi_{m}^{*},\pi_{m-1}^{*}] \colon H^{-s}(M,S^{m}(T^{*}M)) \times H^{-s}(M,S^{m-1}(T^{*}M)) \to H^{-s}(SM), \]
    is bounded for all $s>0$. Equivalently, 
    \[ \tau_{m} \Pi  \colon H^{s}(SM) \to H^{s}(M,S^{m}(T^{*}M)) \times H^{s}(M,S^{m-1}(T^{*}M)), \]
    is bounded for all $s>0$.
\end{prop}

\begin{proof}
    We will prove the second statement. For this, take $h \in H^{s}(SM)$. By Lemma \ref{lemma:asino_prop}, we have that $h \in \mathcal{H}_{+}^{s}$. By the mapping properties of the resolvent (Lemma \ref{lemma:resolvent_aniso}), we have $R_{+}(0)h \in \mathcal{H}_{+}^{s}$. Here we take the order function $m$ equal to $-1$ in a small conic neighborhood of $E_{u}^{*}$, and equal to $1$ outside a bigger but still small neighborhood. In particular, is 1 in $\mathbb{H}^{*} \oplus E_{0}^{*}$. Recall by \eqref{eq:wf-pi0-star} that $\WF(\pi_{m}^{*} \bullet ) \subset \mathbb{F}^{*} \oplus \mathbb{H}^{*}$. So, $\pi_{m}^{*}R_{+}(0) h$ has the same regularity as $R_{+}(0)h$, which is microlocally $H^{s}$. A similar argument shows that $\pi_{m}^{*}R_{-}(0) h \in \mathcal{H}_{+}^{s}$. Indeed, just as before we have $R_{-}(0)h \in \mathcal{H}_{-}^{s}$. Now take an order function been $1$ in a small neighborhood of $E_{s}^{*}$, and $-1$ outside a bigger neighborhood of $E_{s}^{*}$. Hence $m|_{\mathbb{H}^{*} \oplus E_{0}^{*}}=-1$ and as before, $\pi_{m}^{*}R_{-}(0)f \in H^{s}$, since $R_{-}(0)f$ is microlocally in $H^{s}$. This shows the mapping property for $\pi_{m}^{*}\Pi$, and the property for $\tau_{m}\Pi$ follows directly.
\end{proof}

\subsection{Coercivity}

In this section we prove a coercivity result for $N_{m}$. In the case of the geodesic flow, this can be used to prove the local rigidity result from \cite{2019-guillarmou-lefeuvre}, see \cites{GKL22, Lefeuvre}. The proof, as usual in this kind of estimates, is by contradiction.

\begin{prop}
    If $I_{m}$ is s-injective, then there exists a constant $C>0$ such that for any $f \in H^{-1/2}(M,S^{m}(T^{*}M)\times S^{m-1}(T^{*}M))$
    \begin{equation}
    \label{eq:coer}
        \langle N_{m}f,f \rangle_{\mathbf{L}^{2}} \geq C \|\pi_{\ker D_{\mu}^{*}}f\|_{\mathbf{H}^{-1/2}}^{2}.
    \end{equation}
\end{prop}

\begin{proof}
    Let $S_{m}(x): S^{m} T_{x}^{*}M \to S^{m} T_{x}^{*}M$ the square root of $\pi_{m *} \pi_{m}^{*} \colon S^{m}T_{x}^{*}M \to S^{m} T_{x}^{*}M$. For each $m$, define $b \in S^{-1 / 2}(T^{*}M, \operatorname{End}(S^{m}T^{*}M))$ by
\[
b_{m}(x, \xi):=\left(\frac{2 \pi}{C_{n, m}}\right)^{1 / 2} \chi(x, \xi)|\xi|^{-1 / 2} S_{m}(x),
\]
where $\chi \in C^{\infty}(T^{*}M)$ vanishes near the zero section in $T^{*}M$ and is equal to 1 for $|\xi|>1$. Define $B:=\mathrm{diag}(\mathrm{Op}(b_{m}),\mathrm{Op}(b_{m-1}))$, where Op is the left quantization on $M$. Then,
\[
N_{m}=\pi_{\ker D_{\mu}^{*}}B^{*}B\pi_{\ker D_{\mu}^{*}}+R,
\]
where $R \in \Psi^{-2}$. Hence, for $f \in H^{-1 / 2}(M, S^{m}(T^{*}M) \times S^{m-1}(T^{*}M))$ we have:
\begin{equation}
\label{eq:normal_coer1}
\langle N_{m}f,f\rangle_{\mathbf{L}^{2}}=\|B \pi_{\ker D_{\mu}^*} f\|_{\mathbf{L}^2}^2+\langle R f, f\rangle_{\mathbf{L}^{2}}.    
\end{equation}
Since $B$ is elliptic, there exist $Q \in \Psi^{1/2}$ and $R' \in \Psi^{-\infty}$ such that $Q B \pi_{\ker D_{\mu}^{*}}=\pi_{\ker D_{\mu}^{*}}+R'$. By the mapping properties of pseudodifferential operators, we conclude that 
\begin{align}
    \|\pi_{\ker D_{\mu}^{*}} f\|_{\mathbf{H}^{-1 / 2}}^{2} & \leq 2(\|Q B \pi_{\ker D_{\mu}^{*}} f\|_{\mathbf{H}^{-1/2}}^2+\|R' f\|_{\mathbf{H}^{-1 / 2}}^{2}) \\
    & \leq C (\|B \pi_{\ker D_{\mu}^{*}} f\|_{\mathbf{L}^2}^2+\|R' f\|_{\mathbf{H}^{-1/2}}^2),
\end{align}
for any $f \in C^{\infty}(M, S^{m}(T^{*}M) \times S^{m-1}(T^{*}M))$.

Now, observe that the result is trivial when $f$ is potential, since the right-hand side of \eqref{eq:coer} is zero in this case. So, we can assume that $f$ is solenoidal, i.e., $\pi_{\ker D_{\mu}^{*}}f=f$. Hence, using the previous inequality together with \eqref{eq:normal_coer1}, we find
\[
\|f\|_{\mathbf{H}^{-1/2}}^{2} \leq C ( \langle N_{m}f,f \rangle_{\mathbf{L}^{2}}- \langle Rf,f \rangle_{\mathbf{L}^{2}}+\|R'f\|_{\mathbf{H}^{-1/2}}^{2} ).
\]
This implies
\begin{equation}
    \label{eq:normal_coer2}
    \|f\|_{\mathbf{H}^{-1/2}}^{2} \leq C ( \langle N_{m}f,f \rangle_{\mathbf{L}^{2}}+ \|Rf\|_{\mathbf{H}^{1/2}} \|f\|_{\mathbf{H}^{-1/2}}+\|R'f\|_{\mathbf{H}^{-1/2}}^{2}).
\end{equation}
Now we argue by contradiction, that is, we assume that \eqref{eq:coer} does not hold. Hence, there exist a sequence $f_{n} \in C^{\infty}(M,S^{m}(T^{*}M) \times S^{m-1}(T^{*}M))$ of solenoidal tensors such that $\|f_{n}\|_{H^{-1/2}}=1$ and satisfying 
\begin{equation}
    \label{eq:normal_coer3}
    \|\sqrt{\Pi_{m}}f_{n}\|_{\mathbf{L}^{2}}^{2} \leq \|f_{n}\|_{\mathbf{H}^{-1/2}}^{2}/n=1/n \to 0.
\end{equation}
Since $R,R'$ are pseudodifferentials of negative order, they are compact. Therefore, up to passing to a subsequence, we can assume that $Rf_{n} \to v_{1}$ in $H^{1/2}$ and $R'f_{n}\to v_{2}$ in $H^{-1/2}$. This together with \eqref{eq:normal_coer2} imply that $(f_{n})_{n}$ is a Cauchy sequence in $H^{-1/2}$. Thus, it converges to some unitary solenoidal element $v_{3} \in H^{-1/2}$. By continuity of $N_{m} \in \Psi^{-1}$ and the mapping properties of pseudodifferentials operators, we find out that $N_{m}f_{n} \to N_{m}v_{3}$ in $H^{1/2}$. Using this in \eqref{eq:normal_coer3} gives $\langle N_{m}v_{3},v_{3} \rangle_{\mathbf{L}^{2}}=0$. Since $v_{3}$ is soleniodal, this equivalent to $\sqrt{N_{m}}v_{3}=0$, and therefore $N_{m}v_{3}=0$. Now, s-injectivity of the ray transform together with Lemma \ref{lemma:inj_equiv} gives $v_{3}=0$, contradicting the fact that $v_{3}$ is unitary.  
\end{proof}

\section{Injectivity and Stability of the magnetic X-ray transform} \label{sec:inj_stab_xray}

\subsection{Injectivity of the magnetic ray transform}

We devote this subsection to show injectivity of the magnetic ray transform action on pairs of 2-tensors and 1-forms, under some geometric conditions, see Theorem \ref{thm:inj-I2}. The proof is based on the methods of \cite{DPSU07} together with \cite{DP08}. 

Before going into some preliminarily results, let us recall the cases where injectivity is known for the X-ray transform for magnetic and thermostats flows. $I_{1}$ is s-injective for any Anosov thermostat ($\lambda$-)flow on closed surfaces \cite{DP07}, and for any Anosov magnetic flow in closed manifolds by \cite{DP08}. In the context of compact manifolds with boundary, the work \cite{DPSU07} has several results. For simple magnetic flows (i.e., for manifold such that the boundary is strictly convex with respect to magnetic geodesics, and such that the magnetic exponential map, defined using the magnetic flow, is a diffeomorphism at each point), they prove that $I_{1}$ is s-injective and that $I_{2}$ is s-injective assuming a geometric condition involving the curvature of the metric and a smallness condition on the Lorentz force (we explain this below). Their results on boundary action rigidity also show that $I_{m}$ s-injective for any analytic magnetic system. On closed surfaces, the results in \cite{JP09} show that $I_{2}$ is s-injective for negatively curved (in the sense of Riemannian manifolds) Gaussian thermostat flows. Furthermore, in \cite{Ainsworth15} it is proven that $I_{2}$ is s-injective for Anosov magnetic systems over surfaces. Finally, in \cite{AZ17}, the authors shown that $I_{m}$ is injective for any $m$ for any negatively curved Gaussian thermostat flow on closed surfaces. 

\subsubsection{Semibasic tensors}

Here we recall basic facts about \emph{semibasic tensors}. We will follow \cite{DPSU07}. To define them, let $\pi \colon T M \setminus \{0\} \to M$ be the base-point projection, and let $\mathscr{B}_S^r M:=\pi^* \mathcal{T}_S^r M$ denote the bundle of \emph{semibasic tensors of degree} $(r,s)$, where $\mathcal{T}_s^r M$ is the bundle of tensors of degree $(r,s)$ over $M$. Sections of the bundles $\mathscr{B}_s^r M$ are called \emph{semibasic tensor fields}, and we denote the space of smooth sections by $C^{\infty} (\mathscr{B}_s^r M )$ (in particular, $C^{\infty} (\mathscr{B}_0^0 M )=C^{\infty}(T M \setminus\{0\})$). 

For a semibasic tensor field $(T_{j_1 \ldots j_s}^{i_1 \ldots i_r})(x, v)$, we define its \emph{ (covariant) horizontal, vertical, and modified horizontal derivatives}
\begin{align*}
    \stackrel{\mathtt{h}}{\nabla}=\nabla_{|} \colon C^{\infty}(\mathscr{B}_{s}^{r}M) \to C^{\infty}(\mathscr{B}_{s+1}^{r}M), \\
    \stackrel{\mathtt{v}}{\nabla}=\nabla_{\cdot} \colon C^{\infty}(\mathscr{B}_{s}^{r}M) \to C^{\infty}(\mathscr{B}_{s+1}^{r}M), \\
    \stackrel{\mathtt{m}}{\nabla}=\nabla_{:} \colon C^{\infty}(\mathscr{B}_{s}^{r}M) \to C^{\infty}(\mathscr{B}_{s+1}^{r}M),
\end{align*}
by
\begin{align*}
(\nabla_{|} T)_{j_1 \ldots j_s k}^{i_1 \ldots i_r}=\nabla_{|k} T_{j_1 \ldots j_s}^{i_1 \ldots i_r}=T_{j_1 \ldots j_s | k}^{i_1 \ldots i_r}= & \frac{\partial}{\partial x^k} T_{j_1 \ldots j_s}^{i_1 \ldots i_r}-\Gamma_{k q}^p v^q \frac{\partial}{\partial v^p} T_{j_1 \ldots j_s}^{i_1 \ldots i_r} \\
& +\sum_{m=1}^r \Gamma_{k p}^{i_m} T_{j_1 \ldots j_s}^{i_1 \ldots i_{m-1} p i_{m+1} \ldots i_r} \\
&-\sum_{m=1}^r \Gamma_{k j_m}^p T_{j_1 \ldots j_{m-1} p j_{m+1} \ldots j_s}^{i_1 \ldots i_r}, \\
(\nabla_{\cdot} T)_{j_1 \ldots j_s k}^{i_1 \ldots i_r}=\nabla_{\cdot k} T_{j_1 \ldots j_s}^{i_1 \ldots i_r}=T_{j_1 \ldots j_s \cdot k}^{i_1 \ldots i_r}=&\frac{\partial}{\partial v^k} T_{j_1 \ldots j_s}^{i_1 \ldots i_r}, \\
(\nabla_{:} T)_{j_1 \ldots j_s k}^{i_1 \ldots i_r}=\nabla_{:k} T_{j_1 \ldots j_s}^{i_1 \ldots i_r}=T_{j_1 \ldots j_s: k}^{i_1 \ldots i_r}=& T_{j_1 \ldots j_s | k}^{i_1 \ldots i_r}+|v| Y_k^j T_{j_1 \ldots j_s \cdot j}^{i_1 \ldots i_r},
\end{align*}
respectively. We also define
\[ \nabla^{|i}=g^{ij}\nabla_{|j}, \quad \nabla^{\cdot i}=g^{ij}\nabla_{\cdot j}, \quad \nabla^{:i}=g^{ij}\nabla_{:j}. \]
The moral of these operators is that the horizontal derivative only takes in count the horizontal directions, while the vertical only the vertical ones, and the modified twist the horizontal one with the Lorentz force.

Given $u \in C^{\infty}(TM \setminus \{0\})$, we define
\[ \mathbf{F}u(x,v)=v^{i} u_{:i}=v^{i}(u_{|i}+|v|Y_{i}^{j}u_{\cdot j}). \]
Observe that $\mathbf{F}$ restricted to the unit tangent bundle is the generator of the magnetic flow $F$.

\subsubsection{Index}

Let $\Lambda$ be the $\R$-vector space of smooth vector fields $Z$ along $\gamma$ with $Z(0)=Z(T)=0$. Define the \emph{index operator} $\mathbb{I} \colon \Lambda \to \R$ by 
\[ \mathbb{I}(Z,Z)=\int_{0}^{T}\{ |\dot{Z}|^{2}-(\mathcal{C}(Z),Z )-( Y(\dot{\gamma}),Z )^{2} \}dt, \]
where 
\[ \mathcal{C}(Z)=R(\dot{\gamma},Z)\dot{\gamma}-Y(\dot{Z})-(\nabla_{Z}Y)(\dot{\gamma}). \]

One can see that $\mathcal{C}$ appears naturally on the equation for Jacobi fields of magnetic geodesics, see for instance \cite{DP08}*{Section 4}.

\begin{lemma}[\cite{DP08}*{Lemma 4.10}] \label{lemma:index}
If $Z$ is orthogonal to $\dot{\gamma}$, then $\mathbb{I}(Z,Z) \geq 0$, with equality if and only if $Z=0$.
\end{lemma}

\subsubsection{The magnetic ray transform on 2-tensors and 1-forms}

For a magnetic system $(M, g, \alpha)$ and $(x, v) \in S M$, put
\[ k_\mu(x, v)=\sup _w \{2 K(x, \sigma_{v, \eta})+( Y(w), v)^2+(n+3)|Y(w)|^2-2((\nabla_w Y)(v), w)\}, \]
where the supremum is taken over all unit vectors $w \in T_x M$ orthogonal to $v$, and $K(x, \sigma_{v,w})$ is the sectional curvature of the 2-plane $\sigma_{v,w}$ spanned by $v$ and $w$. Define
\[ k_\mu^{+}(x, v)=\max \{0, k(x, v)\}, \]
and
\[ k(M, g, \alpha)=\sup _\gamma T_\gamma \int_0^{T_\gamma} k_\mu^{+}(\gamma(t), \dot{\gamma}(t)) d t, \]
where the supremum is taken over all closed magnetic geodesics $\gamma \colon [0,T_{\gamma}] \to M$. We will need the following result.

\begin{lemma}[\cite{DPSU07}*{Lemma 5.6}] \label{lemma:k_mu}
Assume that $k_{\mu} \leq 4$. Let $\gamma \colon [0, T] \to M$ be a unit speed magnetic geodesic. Then for every smooth vector field $Z$ along $\gamma$ vanishing at the endpoints of $\gamma$ and orthogonal to $\dot{\gamma}$ we have
\[
\int_{0}^{T}\{|\dot{Z}|^2-2 (\mathcal{C}(Z), Z )-( Y(Z), \dot{\gamma})^2-(n+2)|Y(Z)|^{2} \} d t \geq 0,
\]
with equality if and only if $Z=0$.
\end{lemma}

\begin{thm} \label{thm:inj-I2}
$I_{2}$ is s-injective on Anosov magnetic systems satisfying $k_{\mu} \leq 4$.
\end{thm}

\begin{proof}
    Take $H=[p,q]$ be a solenoidal tensor in the kernel of $I_{2}$. We need to show that $p$ and $q$ are zero. By the standard argument, we obtain $u$ with $Fu=\pi_{2}^{*}p+\pi_{1}^{*}q$. In first place, we have the following inequality
\begin{equation} \label{eq:int_curv_mag}
    \int_{SM}\{ |\mathbf{F}u|^{2}-2 (\mathcal{C}(\stackrel{\mathtt{h}}{\nabla}u),\stackrel{\mathtt{h}}{\nabla}u ) -( Y(\stackrel{\mathtt{h}}{\nabla}u),v )^{2}-(n+2)|Y(\stackrel{\mathtt{h}}{\nabla}u)|^{2} \}d\Sigma^{2n-1} \leq 0.
\end{equation}
Here $d\Sigma^{2n-1}$ denotes the Liouville measure in $SM$. This is just \cite{DPSU07}*{Equation (5.23)}, which follows by algebraic manipulations by the integration of certain Pestov identities obtained in \cite{DP08}. On the other hand, Lemma \ref{lemma:k_mu} implies that over any closed magnetic geodesic $\gamma \colon [0,T] \to M$:
\begin{equation} \label{eq:int_index}
    \int_{0}^{T}\{ |\mathbf{F}u|^{2}-2(\mathcal{C}(\stackrel{\mathtt{h}}{\nabla}u),\stackrel{\mathtt{h}}{\nabla}u) -(Y(\stackrel{\mathtt{h}}{\nabla}u),v )^{2}-(n+2)|Y(\stackrel{\mathtt{h}}{\nabla}u)|^{2} \}dt\geq 0,
\end{equation}
on any unit speed magnetic geodesic. By the non-negative Liv\v sic theorem (Lemma \ref{lemma:nnl}), there exists $U$ H\"older so that for any $x$ and any $s$ we have 
\begin{equation} \label{eq:NNL}
    \int_{0}^{s}\{ |\mathbf{F}u|^{2}-2( \mathcal{C}(\stackrel{\mathtt{h}}{\nabla}u),\stackrel{\mathtt{h}}{\nabla}u ) - ( Y(\stackrel{\mathtt{h}}{\nabla}u),v )^{2}-(n+2)|Y(\stackrel{\mathtt{h}}{\nabla}u)|^{2} \}dt-U(\varphi_{s}x)-U(x) \geq 0.
\end{equation}
It follows from \eqref{eq:int_curv_mag} and that the Liouville measure is invariant under the magnetic flow that the integral over $SM$ of \eqref{eq:NNL} is zero. Hence, we obtain equality in \eqref{eq:NNL}, and taking $s=T$, we find equality in \eqref{eq:int_index}. The Index Lemma (Lemma \ref{lemma:index}) now gives that $\stackrel{\mathtt{h}}{\nabla}u=0$. Hence, since $u$ is independent of $v$, we can write $u=f \circ \pi$ where $f$ is a smooth function on the base manifold $M$. Using the fact that $d\pi(F)=v$, the chain rule implies that $Fu=(\pi_{1}^{*}df)(v)$. This shows that $p=0$ and $q=df$. To prove that $q=0$, we use that the pair is solenoidal: Lemma \ref{lemma:adjoint} (see also Remark \ref{rmk:div_m=2}) implies that $\Delta f=0$. Since $M$ is closed, this gives that $f$ is constant, and therefore $q=0$. 
\end{proof}

\begin{rmk}
    We point out that the hypothesis about $k_{\mu}$ holds when $(M,g)$ is negatively curved and the $C^{2}$ norm of $\alpha$ is small enough.
\end{rmk}

\subsection{Stability}

Our last task is to prove the stability of the magnetic X-ray transform that we promised in the introduction.

\begin{proof} [Proof of Theorem \ref{thm:stability_xray}]
    In first place, observe that we can assume, without loss of generality, that $f$ is solenoidal. As usual, let $f=[p,q]$. By the Approximate Liv\v sic Theorem (Lemma \ref{lemma:approx}), we can write 
\begin{equation} \label{eq:ap_liv}
 \pi_{m}^{*}p+\pi_{m-1}^{*}q=Fu+h   
\end{equation}
where 
\begin{equation} \label{eq:ap_liv_bound}
\|h\|_{C^{\beta}}\leq C\|I_{m}f\|_{\ell^{\infty} (\mathcal{C})}^{\tau} \|\pi_{m}^{*}p+\pi_{m-1}^{*}q\|_{C^{1}}^{1-\tau},
\end{equation}
for some $\beta,C,\tau>0$. Then, for $s<\beta$:
\begin{align*}
    \|f\|_{\mathbf{H}^{s-1}} &\lesssim \| N_{m}f\|_{\mathbf{H}^{s}} \\
    & = \left\| \begin{pmatrix}
    \pi_{m *}(\Pi+1\otimes 1) (\pi_{m}^{*}p+\pi_{m-1}^{*}q) \\ \pi_{m-1 *}(\Pi+1\otimes 1) (\pi_{m}^{*}p+\pi_{m-1}^{*}q) \end{pmatrix}  \right\|_{\mathbf{H}^{s}} \\
    & = \left\| \begin{pmatrix}
    \pi_{m *}(\Pi+1\otimes 1) (Fu+h) \\ \pi_{m-1 *}(\Pi+1\otimes 1) (Fu+h) \end{pmatrix}  \right\|_{\mathbf{H}^{s}} \\
    & = \left\| \begin{pmatrix}
    \pi_{m *}(\Pi+1\otimes 1) h \\ \pi_{m-1 *}(\Pi+1\otimes 1) h \end{pmatrix}  \right\|_{\mathbf{H}^{s}} \\
    & \lesssim \|h\|_{H^{s}} \\
    & \lesssim \|h\|_{C^{\beta}} \\
    & \lesssim \|I_{m}f\|_{\ell^{\infty} (\mathcal{C})}^{\tau} \|\pi_{m}^{*}p+\pi_{m-1}^{*}q\|_{C^{1}}^{1-\tau} \\
    & \lesssim \|I_{m}f\|_{\ell^{\infty} (\mathcal{C})}^{\tau} \|f\|_{C^{1}}^{1-\tau},
\end{align*}
where in the first step we used Proposition \ref{prop:cont}, in the third step we used \eqref{eq:ap_liv}, in the fourth we used that $\Pi F=0$ (Lemma \ref{lemma:pi_properties} (i)) and the fact that the integral of $Fh$ vanishes (which follows from Stokes' theorem, the fact that $M$ is closed and that $F$ preservs the Liouville form), in the fifth step we used Proposition \ref{prop:bound_tau}, and in the seventh we used \eqref{eq:ap_liv_bound}.
\end{proof}

\bibliographystyle{alpha}
\bibliography{bib}

@article {Ainsworth15,
    AUTHOR = {Ainsworth, Gareth},
     TITLE = {The magnetic ray transform on {A}nosov surfaces},
   JOURNAL = {Discrete Contin. Dyn. Syst.},
  FJOURNAL = {Discrete and Continuous Dynamical Systems},
    VOLUME = {35},
      YEAR = {2015},
    NUMBER = {5},
     PAGES = {1801--1816},
      ISSN = {1078-0947,1553-5231},
   MRCLASS = {53D25 (53C21 53C22)},
  MRNUMBER = {3294225},
MRREVIEWER = {Boris\ Hasselblatt},
       DOI = {10.3934/dcds.2015.35.1801},
       URL = {https://doi.org/10.3934/dcds.2015.35.1801},
}

@book {1967-anosov,
    AUTHOR = {Anosov, D. V.},
     TITLE = {Geodesic flows on closed {R}iemann manifolds with negative
              curvature},
    SERIES = {Proceedings of the Steklov Institute of Mathematics},
    VOLUME = {No. 90 (1967)},
      NOTE = {Translated from the Russian by S. Feder},
 PUBLISHER = {American Mathematical Society, Providence, RI},
      YEAR = {1969},
     PAGES = {iv+235},
   MRCLASS = {57.50 (53.00)},
  MRNUMBER = {242194},
}

@article{AdSMRT24,
      title={Marked length spectrum rigidity for Anosov magnetic surfaces}, 
      author={Valerio Assenza and Jacopo de Simoi and James Marshall Reber and Ivo Terek},
      year={2024},
      eprint={2409.20545},
      archivePrefix={arXiv},
      primaryClass={math.DG},
      url={https://arxiv.org/abs/2409.20545}, 
}

@article {AZ17,
    AUTHOR = {Assylbekov, Yernat M. and Zhou, Hanming},
     TITLE = {Invariant distributions and tensor tomography for {G}aussian
              thermostats},
   JOURNAL = {Comm. Anal. Geom.},
  FJOURNAL = {Communications in Analysis and Geometry},
    VOLUME = {25},
      YEAR = {2017},
    NUMBER = {5},
     PAGES = {895--926},
      ISSN = {1019-8385,1944-9992},
   MRCLASS = {53C65 (35R30 82C05)},
  MRNUMBER = {3733794},
MRREVIEWER = {Chunna\ Zeng},
       DOI = {10.4310/CAG.2017.v25.n5.a1},
       URL = {https://doi.org/10.4310/CAG.2017.v25.n5.a1},
}

@article {1985-burns-katok,
    AUTHOR = {Burns, K. and Katok, A.},
     TITLE = {Manifolds with nonpositive curvature},
   JOURNAL = {Ergodic Theory Dynam. Systems},
  FJOURNAL = {Ergodic Theory and Dynamical Systems},
    VOLUME = {5},
      YEAR = {1985},
    NUMBER = {2},
     PAGES = {307--317},
      ISSN = {0143-3857,1469-4417},
   MRCLASS = {53C20 (53-02 58F15)},
  MRNUMBER = {796758},
MRREVIEWER = {Wolfgang\ Ziller},
       DOI = {10.1017/S0143385700002935},
       URL = {https://doi.org/10.1017/S0143385700002935},
}

@article {2022-cekic-paternain,
    AUTHOR = {Ceki\'c, Mihajlo and Paternain, Gabriel P},
     TITLE = {Resonant forms at zero for dissipative {A}nosov flows},
   JOURNAL = {Geom. Topol.},
  FJOURNAL = {Geometry \& Topology},
    VOLUME = {29},
      YEAR = {2025},
    NUMBER = {7},
     PAGES = {3635--3716},
      ISSN = {1465-3060,1364-0380},
   MRCLASS = {99-06},
  MRNUMBER = {4974989},
       DOI = {10.2140/gt.2025.29.3635},
       URL = {https://doi.org/10.2140/gt.2025.29.3635},
}

@article {CIPP00,
    AUTHOR = {Contreras, G. and Iturriaga, R. and Paternain, G. P. and
              Paternain, M.},
     TITLE = {The {P}alais-{S}male condition and {M}a\~n\'e's critical
              values},
   JOURNAL = {Ann. Henri Poincar\'e},
  FJOURNAL = {Annales Henri Poincar\'e. A Journal of Theoretical and
              Mathematical Physics},
    VOLUME = {1},
      YEAR = {2000},
    NUMBER = {4},
     PAGES = {655--684},
      ISSN = {1424-0637,1424-0661},
   MRCLASS = {37J50 (37J45 58E10)},
  MRNUMBER = {1785184},
MRREVIEWER = {Vittorio\ Coti Zelati},
       DOI = {10.1007/PL00001011},
       URL = {https://doi.org/10.1007/PL00001011},
}

@article {DP07,
    AUTHOR = {Dairbekov, Nurlan S. and Paternain, Gabriel P.},
     TITLE = {Entropy production in {G}aussian thermostats},
   JOURNAL = {Comm. Math. Phys.},
  FJOURNAL = {Communications in Mathematical Physics},
    VOLUME = {269},
      YEAR = {2007},
    NUMBER = {2},
     PAGES = {533--543},
      ISSN = {0010-3616,1432-0916},
   MRCLASS = {37D20 (37A60 37D40 82C05)},
  MRNUMBER = {2274556},
MRREVIEWER = {Boris\ Hasselblatt},
       DOI = {10.1007/s00220-006-0117-y},
       URL = {https://doi.org/10.1007/s00220-006-0117-y},
}

@article {DP08,
    AUTHOR = {Dairbekov, Nurlan S. and Paternain, Gabriel P.},
     TITLE = {Rigidity properties of {A}nosov optical hypersurfaces},
   JOURNAL = {Ergodic Theory Dynam. Systems},
  FJOURNAL = {Ergodic Theory and Dynamical Systems},
    VOLUME = {28},
      YEAR = {2008},
    NUMBER = {3},
     PAGES = {707--737},
      ISSN = {0143-3857,1469-4417},
   MRCLASS = {37J05 (37A20 37C20 37C40 37D20 53C24 53C60)},
  MRNUMBER = {2422013},
MRREVIEWER = {C\'esar\ J.\ Niche},
       DOI = {10.1017/S0143385707000612},
       URL = {https://doi.org/10.1017/S0143385707000612},
}

@article {DPSU07,
    AUTHOR = {Dairbekov, Nurlan S. and Paternain, Gabriel P. and Stefanov,
              Plamen and Uhlmann, Gunther},
     TITLE = {The boundary rigidity problem in the presence of a magnetic
              field},
   JOURNAL = {Adv. Math.},
  FJOURNAL = {Advances in Mathematics},
    VOLUME = {216},
      YEAR = {2007},
    NUMBER = {2},
     PAGES = {535--609},
      ISSN = {0001-8708,1090-2082},
   MRCLASS = {37J50 (53C24 53D25 58E10)},
  MRNUMBER = {2351370},
MRREVIEWER = {Rodney\ Josu\'e\ Biezuner},
       DOI = {10.1016/j.aim.2007.05.014},
       URL = {https://doi.org/10.1016/j.aim.2007.05.014},
}

@article {2016-dyatlov-zworski,
    AUTHOR = {Dyatlov, Semyon and Zworski, Maciej},
     TITLE = {Dynamical zeta functions for {A}nosov flows via microlocal
              analysis},
   JOURNAL = {Ann. Sci. \'Ec. Norm. Sup\'er. (4)},
  FJOURNAL = {Annales Scientifiques de l'\'Ecole Normale Sup\'erieure.
              Quatri\`eme S\'erie},
    VOLUME = {49},
      YEAR = {2016},
    NUMBER = {3},
     PAGES = {543--577},
      ISSN = {0012-9593,1873-2151},
   MRCLASS = {37C30 (37D20 58J20)},
  MRNUMBER = {3503826},
MRREVIEWER = {Leonid\ Friedlander},
       DOI = {10.24033/asens.2290},
       URL = {https://doi.org/10.24033/asens.2290},
}

@article {LlMM_smooth_livsic,
    AUTHOR = {de la Llave, R. and Marco, J. M. and Moriy\'on, R.},
     TITLE = {Canonical perturbation theory of {A}nosov systems and
              regularity results for the {L}iv\v sic cohomology equation},
   JOURNAL = {Ann. of Math. (2)},
  FJOURNAL = {Annals of Mathematics. Second Series},
    VOLUME = {123},
      YEAR = {1986},
    NUMBER = {3},
     PAGES = {537--611},
      ISSN = {0003-486X,1939-8980},
   MRCLASS = {58F15 (34C35 58F30)},
  MRNUMBER = {840722},
MRREVIEWER = {Tong\ Ren\ Ding},
       DOI = {10.2307/1971334},
       URL = {https://doi.org/10.2307/1971334},
}

@misc{ECMR25,
      title={Thermostats without conjugate points}, 
      author={Javier {Echevarría Cuesta} and James {Marshall Reber}},
      year={2025},
      eprint={2501.01923},
      archivePrefix={arXiv},
      primaryClass={math.DS},
      url={https://arxiv.org/abs/2501.01923}, 
}

@article {2011-faure-sjostrand,
    AUTHOR = {Faure, Fr\'ed\'eric and Sj\"ostrand, Johannes},
     TITLE = {Upper bound on the density of {R}uelle resonances for {A}nosov
              flows},
   JOURNAL = {Comm. Math. Phys.},
  FJOURNAL = {Communications in Mathematical Physics},
    VOLUME = {308},
      YEAR = {2011},
    NUMBER = {2},
     PAGES = {325--364},
      ISSN = {0010-3616,1432-0916},
   MRCLASS = {37C30 (37D20 58J42 81Q20)},
  MRNUMBER = {2851145},
MRREVIEWER = {Miaohua\ Jiang},
       DOI = {10.1007/s00220-011-1349-z},
       URL = {https://doi.org/10.1007/s00220-011-1349-z},
}

@article {FSU08,
    AUTHOR = {Frigyik, Bela and Stefanov, Plamen and Uhlmann, Gunther},
     TITLE = {The {X}-ray transform for a generic family of curves and
              weights},
   JOURNAL = {J. Geom. Anal.},
  FJOURNAL = {Journal of Geometric Analysis},
    VOLUME = {18},
      YEAR = {2008},
    NUMBER = {1},
     PAGES = {89--108},
      ISSN = {1050-6926,1559-002X},
   MRCLASS = {53C65},
  MRNUMBER = {2365669},
MRREVIEWER = {Rodney\ Josu\'e\ Biezuner},
       DOI = {10.1007/s12220-007-9007-6},
       URL = {https://doi.org/10.1007/s12220-007-9007-6},
}

@article {GL21,
    AUTHOR = {Gou\"ezel, S\'ebastien and Lefeuvre, Thibault},
     TITLE = {Classical and microlocal analysis of the x-ray transform on
              {A}nosov manifolds},
   JOURNAL = {Anal. PDE},
  FJOURNAL = {Analysis \& PDE},
    VOLUME = {14},
      YEAR = {2021},
    NUMBER = {1},
     PAGES = {301--322},
      ISSN = {2157-5045,1948-206X},
   MRCLASS = {37C27 (37D40 53C21 53C22 53C65)},
  MRNUMBER = {4229205},
MRREVIEWER = {Emmanuel\ Schenck},
       DOI = {10.2140/apde.2021.14.301},
       URL = {https://doi.org/10.2140/apde.2021.14.301},
}

@article {Guillarmou17,
    AUTHOR = {Guillarmou, Colin},
     TITLE = {Invariant distributions and {X}-ray transform for {A}nosov
              flows},
   JOURNAL = {J. Differential Geom.},
  FJOURNAL = {Journal of Differential Geometry},
    VOLUME = {105},
      YEAR = {2017},
    NUMBER = {2},
     PAGES = {177--208},
      ISSN = {0022-040X,1945-743X},
   MRCLASS = {53D25 (37D20 53C65)},
  MRNUMBER = {3606728},
MRREVIEWER = {Thomas\ Barthelm\'e},
       DOI = {10.4310/jdg/1486522813},
       URL = {https://doi.org/10.4310/jdg/1486522813},
}

@article {GKL22,
    AUTHOR = {Guillarmou, Colin and Knieper, Gerhard and Lefeuvre, Thibault},
     TITLE = {Geodesic stretch, pressure metric and marked length spectrum
              rigidity},
   JOURNAL = {Ergodic Theory Dynam. Systems},
  FJOURNAL = {Ergodic Theory and Dynamical Systems},
    VOLUME = {42},
      YEAR = {2022},
    NUMBER = {3},
     PAGES = {974--1022},
      ISSN = {0143-3857,1469-4417},
   MRCLASS = {37D40 (37C27 37D35)},
  MRNUMBER = {4374964},
       DOI = {10.1017/etds.2021.75},
       URL = {https://doi.org/10.1017/etds.2021.75},
}

@article {2019-guillarmou-lefeuvre,
    AUTHOR = {Guillarmou, Colin and Lefeuvre, Thibault},
     TITLE = {The marked length spectrum of {A}nosov manifolds},
   JOURNAL = {Ann. of Math. (2)},
  FJOURNAL = {Annals of Mathematics. Second Series},
    VOLUME = {190},
      YEAR = {2019},
    NUMBER = {1},
     PAGES = {321--344},
      ISSN = {0003-486X,1939-8980},
   MRCLASS = {53C24 (37C27 37D40 53C22)},
  MRNUMBER = {3990606},
MRREVIEWER = {Rafael\ Oswaldo\ Ruggiero},
       DOI = {10.4007/annals.2019.190.1.6},
       URL = {https://doi.org/10.4007/annals.2019.190.1.6},
}

@article {JP09,
    AUTHOR = {Jane, Dan and Paternain, Gabriel P.},
     TITLE = {On the injectivity of the {X}-ray transform for {A}nosov
              thermostats},
   JOURNAL = {Discrete Contin. Dyn. Syst.},
  FJOURNAL = {Discrete and Continuous Dynamical Systems},
    VOLUME = {24},
      YEAR = {2009},
    NUMBER = {2},
     PAGES = {471--487},
      ISSN = {1078-0947,1553-5231},
   MRCLASS = {37D40 (53C65)},
  MRNUMBER = {2486586},
MRREVIEWER = {Jos\'e\ Ant\^onio G. Miranda},
       DOI = {10.3934/dcds.2009.24.471},
       URL = {https://doi.org/10.3934/dcds.2009.24.471},
}

@article{1974-klingenberg-riemannian,
    AUTHOR = {Klingenberg, Wilhelm},
     TITLE = {Riemannian manifolds with geodesic flow of {A}nosov type},
   JOURNAL = {Ann. of Math. (2)},
  FJOURNAL = {Annals of Mathematics. Second Series},
    VOLUME = {99},
      YEAR = {1974},
     PAGES = {1--13},
      ISSN = {0003-486X},
   MRCLASS = {58E10 (58F15)},
  MRNUMBER = {377980},
MRREVIEWER = {Robert\ Roussarie},
       DOI = {10.2307/1971011},
       URL = {https://doi.org/10.2307/1971011},
}

@book{Lefeuvre  ,
   author={Lefeuvre, Thibault},
   title={Microlocal analysis in hyperbolic dynamics and geometry},
   note={\url{https://thibaultlefeuvre.blog/microlocal-analysis-in-hyperbolic-dynamics-and-geometry-2/}}
}

@article {LT05,
    AUTHOR = {Lopes, A. O. and Thieullen, Ph.},
     TITLE = {Sub-actions for {A}nosov flows},
   JOURNAL = {Ergodic Theory Dynam. Systems},
  FJOURNAL = {Ergodic Theory and Dynamical Systems},
    VOLUME = {25},
      YEAR = {2005},
    NUMBER = {2},
     PAGES = {605--628},
      ISSN = {0143-3857,1469-4417},
   MRCLASS = {37D20},
  MRNUMBER = {2129112},
MRREVIEWER = {Oliver\ Jenkinson},
       DOI = {10.1017/S0143385704000732},
       URL = {https://doi.org/10.1017/S0143385704000732},
}

@article{1987-mane-on-a-theorem-of-klingenberg,
    AUTHOR = {Ma\~n\'e, R.},
     TITLE = {On a theorem of {K}lingenberg},
 BOOKTITLE = {Dynamical systems and bifurcation theory ({R}io de {J}aneiro,
              1985)},
    VOLUME = {160},
     PAGES = {319--345},
 PUBLISHER = {Longman Sci. Tech., Harlow},
      YEAR = {1987},
      ISBN = {0-582-44261-3},
   MRCLASS = {58F17 (53C22 58F15)},
  MRNUMBER = {907897},
MRREVIEWER = {Victor\ Bangert},
}

@article {OPS25,
    AUTHOR = {Oksanen, Lauri and Paternain, Gabriel P. and Sarkkinen, Miika},
     TITLE = {On the {I}nterplay between the {L}ight {R}ay and the
              {M}agnetic {X}-{R}ay {T}ransforms},
   JOURNAL = {SIAM J. Math. Anal.},
  FJOURNAL = {SIAM Journal on Mathematical Analysis},
    VOLUME = {57},
      YEAR = {2025},
    NUMBER = {6},
     PAGES = {6522--6541},
      ISSN = {0036-1410,1095-7154},
   MRCLASS = {53 (35R30 58)},
  MRNUMBER = {4988343},
       DOI = {10.1137/25M173243X},
       URL = {https://doi.org/10.1137/25M173243X},
}

@book {Paternain,
    AUTHOR = {Paternain, Gabriel P.},
     TITLE = {Geodesic flows},
    SERIES = {Progress in Mathematics},
    VOLUME = {180},
 PUBLISHER = {Birkh\"auser Boston, Inc., Boston, MA},
      YEAR = {1999},
     PAGES = {xiv+149},
      ISBN = {0-8176-4144-0},
   MRCLASS = {53D25 (37D40 37J99)},
  MRNUMBER = {1712465},
MRREVIEWER = {Boris\ Hasselblatt},
       DOI = {10.1007/978-1-4612-1600-1},
       URL = {https://doi.org/10.1007/978-1-4612-1600-1},
}

@article {1994-paternain-paternain,
    AUTHOR = {Paternain, Gabriel P. and Paternain, Miguel},
     TITLE = {On {A}nosov energy levels of convex {H}amiltonian systems},
   JOURNAL = {Math. Z.},
  FJOURNAL = {Mathematische Zeitschrift},
    VOLUME = {217},
      YEAR = {1994},
    NUMBER = {3},
     PAGES = {367--376},
      ISSN = {0025-5874,1432-1823},
   MRCLASS = {58F15 (58F05)},
  MRNUMBER = {1306666},
MRREVIEWER = {Edoh\ Amiran},
       DOI = {10.1007/BF02571949},
       URL = {https://doi.org/10.1007/BF02571949},
}

@article {PSU14,
    AUTHOR = {Paternain, Gabriel P. and Salo, Mikko and Uhlmann, Gunther},
     TITLE = {Spectral rigidity and invariant distributions on {A}nosov
              surfaces},
   JOURNAL = {J. Differential Geom.},
  FJOURNAL = {Journal of Differential Geometry},
    VOLUME = {98},
      YEAR = {2014},
    NUMBER = {1},
     PAGES = {147--181},
      ISSN = {0022-040X,1945-743X},
   MRCLASS = {37D20 (53D25)},
  MRNUMBER = {3263517},
MRREVIEWER = {Boris\ Hasselblatt},
       URL = {http://projecteuclid.org/euclid.jdg/1406137697},
}

@article {2015-paternain-salo-uhlmann-invariant,
    AUTHOR = {Paternain, Gabriel P. and Salo, Mikko and Uhlmann, Gunther},
     TITLE = {Invariant distributions, {B}eurling transforms and tensor
              tomography in higher dimensions},
   JOURNAL = {Math. Ann.},
  FJOURNAL = {Mathematische Annalen},
    VOLUME = {363},
      YEAR = {2015},
    NUMBER = {1-2},
     PAGES = {305--362},
      ISSN = {0025-5831,1432-1807},
   MRCLASS = {53D25 (58J05)},
  MRNUMBER = {3394381},
MRREVIEWER = {B.\ S.\ Rubin},
       DOI = {10.1007/s00208-015-1169-0},
       URL = {https://doi.org/10.1007/s00208-015-1169-0},
}

@book{PSU23,
    AUTHOR = {Paternain, Gabriel P. and Salo, Mikko and Uhlmann, Gunther},
     TITLE = {Geometric inverse problems---with emphasis on two dimensions},
    SERIES = {Cambridge Studies in Advanced Mathematics},
    VOLUME = {204},
      NOTE = {With a foreword by Andr\'as Vasy},
 PUBLISHER = {Cambridge University Press, Cambridge},
      YEAR = {2023},
     PAGES = {xxiv+344},
      ISBN = {978-1-316-51087-2},
   MRCLASS = {35-02 (35R30 53C65 58J32 92C55)},
  MRNUMBER = {4520155},
}

@article {SSU05,
    AUTHOR = {Sharafutdinov, Vladimir and Skokan, Michal and Uhlmann,
              Gunther},
     TITLE = {Regularity of ghosts in tensor tomography},
   JOURNAL = {J. Geom. Anal.},
  FJOURNAL = {The Journal of Geometric Analysis},
    VOLUME = {15},
      YEAR = {2005},
    NUMBER = {3},
     PAGES = {499--542},
      ISSN = {1050-6926,1559-002X},
   MRCLASS = {58J40 (44A12 53C65)},
  MRNUMBER = {2190243},
MRREVIEWER = {Leonid\ N.\ Pestov},
       DOI = {10.1007/BF02930983},
       URL = {https://doi.org/10.1007/BF02930983},
}

@article {2004-stefanov-uhlmann,
    AUTHOR = {Stefanov, Plamen and Uhlmann, Gunther},
     TITLE = {Stability estimates for the {X}-ray transform of tensor fields
              and boundary rigidity},
   JOURNAL = {Duke Math. J.},
  FJOURNAL = {Duke Mathematical Journal},
    VOLUME = {123},
      YEAR = {2004},
    NUMBER = {3},
     PAGES = {445--467},
      ISSN = {0012-7094,1547-7398},
   MRCLASS = {53C65 (44A12 47G30 53C24)},
  MRNUMBER = {2068966},
MRREVIEWER = {Dorothee\ Schueth},
       DOI = {10.1215/S0012-7094-04-12332-2},
       URL = {https://doi.org/10.1215/S0012-7094-04-12332-2},
}

@article {1984-ghys,
    AUTHOR = {Ghys, \'Etienne},
     TITLE = {Flots d'{A}nosov sur les {$3$}-vari\'et\'es fibr\'ees en
              cercles},
   JOURNAL = {Ergodic Theory Dynam. Systems},
  FJOURNAL = {Ergodic Theory and Dynamical Systems},
    VOLUME = {4},
      YEAR = {1984},
    NUMBER = {1},
     PAGES = {67--80},
      ISSN = {0143-3857,1469-4417},
   MRCLASS = {58F15 (58F17)},
  MRNUMBER = {758894},
MRREVIEWER = {Alberto\ Verjovsky},
       DOI = {10.1017/S0143385700002273},
       URL = {https://doi-org.offcampus.lib.washington.edu/10.1017/S0143385700002273},
}

\end{document}